\documentclass[12pt, reqno]{amsart}
\usepackage{amssymb}
\usepackage{graphicx}
\usepackage{physics}
\usepackage{amssymb}
 \usepackage{amsthm}
\usepackage{epsf}
\usepackage{amsbsy,amsmath}
\usepackage{mathtools}
\usepackage{upgreek}
\usepackage{mathrsfs}
\usepackage{amsfonts}
\usepackage{amssymb}
\usepackage{float}
\usepackage{enumitem}
\usepackage{eucal}
\usepackage{graphics,mathrsfs}
\usepackage{amsthm}
\usepackage{secdot}
\usepackage{esint}
\usepackage{hyperref}
\usepackage{varwidth}
\usepackage{tasks}
\usepackage{abstract}
\usepackage{orcidlink}
\theoremstyle{plain}
\newtheorem{theorem}{Theorem}[section]
\newtheorem{cor}{Corollary}[section]

\theoremstyle{definition}
\newtheorem{definition}{Definition}[section] % This numbers definitions based on sections
\newtheorem{lemma}{Lemma}[section]
\newtheorem{remark}{Remark}[section]

\allowdisplaybreaks

\usepackage{xcolor}
\usepackage{bbm}
\usepackage{cite}
\usepackage{comment}

\def\dist{{\operatorname{dist}}}
\def\div{{\operatorname{div}}}
\def\supp{{\operatorname{supp}}}
\renewcommand{\d}{\:\! \mathrm{d}}

\long\def\symbolfootnote[#1]#2{\begingroup
\def\thefootnote{\fnsymbol{footnote}}\footnote[#1]{#2}\endgroup}
\numberwithin{equation}{section}
\begin{document}
\title[On sign-changing solutions for mixed local-nonlocal $p$-Laplace operator] {On sign-changing solutions for mixed local and nonlocal $p$-Laplace operator}

\author[S. Bhowmick, S. Ghosh]{Souvik Bhowmick$^{1}$\orcidlink{0009-0009-9661-7219} and Sekhar Ghosh$^{1}$\orcidlink{0000-0002-5082-2374}
}

\iffalse
\author{Souvik Bhowmick}
\address[Souvik Bhowmick]{Department of Mathematics, National Institute of Technology Calicut, Kozhikode, Kerala, India - 673601}
\email{souvikbhowmick2912@gmail.com / souvik\_p230197ma@nitc.ac.in}
\author{Sekhar Ghosh}
\address[ Sekhar Ghosh]{Department of Mathematics, National Institute of Technology Calicut, Kozhikode, Kerala, India - 673601}
\email{sekharghosh1234@gmail.com / sekharghosh@nitc.ac.in}
\fi

\maketitle

\begin{center}
		$^1$Department of Mathematics, National Institute of Technology Calicut,\\
	Kozhikode - 673601, Kerala, India.\\
    \textit{Email addresses:} souvikbhowmick2912@gmail.com, sekharghosh1234@gmail.com
\end{center}
\vspace{0.3cm}
\begin{abstract}
In this paper, we use the method of invariant sets of descending flows to demonstrate the existence of multiple sign-changing solutions for a class of elliptic problems with zero Dirichlet boundary conditions. By combining Nehari manifold techniques with a constrained variational approach and Brouwer degree theory, we establish the existence of a least-energy sign-changing solution. Furthermore, we prove that the energy of the least energy sign-changing solution is strictly greater than twice the ground state energy. This work extends the celebrated results of Bartsch $et~al.$ [Proc. Lond. Math. Soc. (3), 91(1): 129–152, 2005] and Chang $et~al.$ [Adv. Nonlinear Stud., 19(1): 29-53, 2019] to the mixed local and nonlocal $p$-Laplace operator, providing a novel contribution even in the case when $p=2$.
\end{abstract}

%\tableofcontents
\section{Introduction and main theorems}\label{sec1}

\noindent In this paper, we study the following mixed local and nonlocal elliptic equation;
\begin{equation}\label{MP}
\begin{aligned}
    -\Delta_p u +(-\Delta)^s_p  u&= f(x,u)   \text{ in } \,\,\Omega,\\
      u &=0 \,\text{in} \,~ \mathbb R^N \setminus \Omega,
   \end{aligned}
    \end{equation}
    where $\Omega\subset\mathbb{R}^N$ is a bounded domain with smooth boundary $\partial \Omega$, $s\in(0,1)$, $p\in(1,N)$, $-\Delta_p u=\div (|\nabla u|^{p-2}\nabla u)$ represents the $p$-Laplace operator and $(-\Delta)^s_p$ denotes the fractional $p$-Laplace operator defined as the Cauchy principle value, which is given by
       $$(-\Delta)^s_p u(x) = C({N,s,p}) P.V \int_{\mathbb{R}^N} \frac{(|u(x)-u(y)|^{p-2})(u(x)-u(y))}{|x-y|^{N+ps}} \, \d y,$$
where $C({N,s,p})$ is a normalizing constant. For simplicity, we use $C({N,s,p})=1$. We impose the following assumptions on the function $f:\bar{\Omega}\times \mathbb{R}\rightarrow \mathbb{R}$;
\begin{itemize}
    \item[($f_1$)] $f\in C(\bar{\Omega}\cross \mathbb{R})$ and $\lim_{|u|\rightarrow 0}\frac{f(x,u)}{|u|^{p-2}u}=0,$ uniformly in $x\in\bar{\Omega}.$
    
    \item[($f_2$)] There exist $C>0$ and $q\in(p,p^*)$ with $p^*=\frac{pN}{N-p}$ such that 
   $$|f(x,u)|\leq C(1+|u|^{q-1}),~\text{for every} ~(x,u)\in\bar{\Omega}\times \mathbb R.$$
   
   \item[($f_3$)]  There exist $\mu>p$ and $M>0$ such that for all $x\in\bar{\Omega}$ and $|u|\geq M$,
   $$f(x,u)u\geq \mu F(x,u)>0,~  \text{where}~ F(x,u):=\int_0^u f(x,\tau)\d \tau.$$
   
   \item[($f_4$)] $\lim\limits_{|u|\rightarrow +\infty}\frac{f(x,u)}{|u|^{p-2}u}=+\infty,$ uniformly in $x\in\bar{\Omega}.$
   
   \item[($f_5$)] The function $\frac{f(x,u)}{|u|^{p-2}u}$ is strictly increasing in $(0,+\infty)$ and strictly decreasing in $(-\infty,0)$.
   
   \item[($f_6$)] For all $x\in \Omega$ and $u\in \mathbb R$, we have $f(x,-u)=-f(x,u)$.
\end{itemize}

\par In the past decades, researchers have extensively examined the existence and multiplicity of solutions to elliptic partial differential equations (PDEs) associated with Laplace, fractional Laplace operators and their nonlinear counterpart. In the pioneering work by Aubin \cite{A1976}, Talenti \cite{T1976} and Gidas, Ni, and Nirenberg \cite{GNN1979} the set of positive solutions was completely classified for the semilinear PDEs involving the Laplacian. In particular, they proved that minimizers of the Sobolev inequality are attained by a family of ``fixed-sign'' solutions (positive or negative). On the other hand, in the study of the Dirichlet eigenvalue problem $-\Delta u=\lambda u$, we know that all higher eigenfunctions change sign except the principal eigenfunction. Dancer and Du investigated sign-changing solutions in a series of studies \cite{DD1995, DD1994, DD1996}, which play a crucial role in the study of the ``Lotka–Volterra" competing species system involving two species. It is noteworthy to mention that sign-changing solutions arise as ``limit problems" to a class boundary value problem with lack of compactness \cite{W2006} and ``ecological problems" \cite{DD1995}. Moreover, sign-changing solutions appear in various fields of applied sciences, viz.  Optimizations, anomalous diffusion, minimal surface, phase transition, flame propagation, finance \cite{DCC1999,FER1978}. Therefore, sign-changing solutions need attention to classify the set of all solutions to an operator, as well as to study the limiting case of certain problems. Off late, researchers have grown significant interests in studying the existence, multiplicity, and regularity of sign-changing solutions to elliptic PDEs involving the Laplacian, $p$-Laplacian, fractional Laplacian, fractional $p$-Laplacian, etc. We refer to \cite{W1991,CCN1997,BL2004,BLW2005,CNW2019,CW2014,BCW2000,LW2004, Z2008} and the reference therein.

\par One of the earliest study of sign-changing solutions (nodal solutions) is due to Wang \cite{W1991}, where the author considered the following problem,
\begin{equation}\label{eq1.1}
\begin{aligned}
    -\Delta u &= f(x,u)   \text{ in } \,\,\Omega,\\
      u &=0 \,\text{in} \,~ \partial \Omega,
   \end{aligned}
    \end{equation}
where $\Omega\subset\mathbb{R}^N$ is a bounded domain with smooth boundary $\partial \Omega$, $N>2$, $f:\bar{\Omega}\cross \mathbb{R}\rightarrow \mathbb{R}$ is continuous. In \cite{W1991}, the author proved that problem \eqref{eq1.1} has one positive, one negative, and one non-trivial solution by employing the linking method and Morse theory without assuming any symmetry. In their celebrated paper Castro $et ~al.$ \cite{CCN1997} extended the results of \cite{W1991}, establishing the existence of one positive solution, one negative solution, and one sign-changing solution under certain conditions on $f(x,u)=f(u)$, by using the direct method along with the variational splitting $$J(u)=J(u^+)+J(u^-)~\text{and}~\gamma(u)=\gamma(u^+)+\gamma(u^-),$$ where $$J(u)=\frac{1}{2}\int_\Omega |\nabla u|^2\d x -\int_\Omega F(u)\d x \text{ and }\gamma(u)=\langle J'(u),u\rangle.$$ 
Bartsch and Wang \cite{BW1996} developed an abstract critical point theory for a functional on partially ordered Hilbert spaces to establish the existence of a sign-changing solution to the problem \eqref{eq1.1} with some weak condition on $f(u)=f(x,u)$. Later Bartsch $et~ al.$\cite{BCW2000} proved the existence of a sign-changing solution of the problem \eqref{eq1.1} by using the Morse index. In \cite{BW2003}, Bartsch and Weth established the existence of a sign-changing solution with the properties of the nodal domains and obtained the location of subsolutions and supersolutions to the problem \eqref{eq1.1}. Liu and Sun \cite{LS2001} introduced the method of invariant sets of descending flow to guarantee the existence of multiple sign-changing solutions to the problem \eqref{eq1.1}. Liu and Wang \cite{LW2004} proved the existence and multiplicity of sign-changing solutions to the problem \eqref{eq1.1} by using the Nehari manifold techniques under some weak conditions on $f$. In \cite{RW2009}, Roselli and Willem established the existence of least energy sign-changing solutions for the Brezis-Nirenberg problem using the Nehari manifold. The following nonlinear extension to the problem \eqref{eq1.1} involving the $p$-Laplacian was studied by Bartsch and Liu \cite{BL2004} to guarantee the existence of sign-changing solutions.
\begin{equation}\label{eq1.2}
\begin{aligned}
    -\Delta_p u &= f(x,u)   \text{ in } \,\,\Omega,\\
      u &=0 \,\text{in} \,~ \partial \Omega,
   \end{aligned}
    \end{equation}
where $\Omega\subset\mathbb{R}^N$ is a bounded domain with smooth boundary $\partial \Omega$, $1<p<N$, $p\in(1,\infty)$, $f:\bar{\Omega}\cross \mathbb{R}\rightarrow \mathbb{R}$ is continuous.
In \cite{BL2004}, the authors established the existence of four solutions by applying a critical point theorem for $C^1$-functionals on partially ordered Banach spaces and have used the method of descending flow whenever \eqref{eq1.2} possesses a subsolution and a supersolution. Furthermore, they established the existence of one positive solution, one negative solution, and one sign-changing solution to the problem \eqref{eq1.2}. Moreover, in \cite{BLW2005}, Bartsch $et~al.$ constructed a new variational approach to establish the existence of a sign-changing solution to \eqref{eq1.2} by developing a critical point theory in Banach spaces. For further detailed discussion on the development of sign-changing solutions, we refer to \cite{BCW2000,BW2003,RW2009,BWW2005,B2001,DD1995,DD1994,DD1996} and the references cited therein. 

\par We now focus on the nonlocal counterpart of the problem \eqref{eq1.1}. Consider the problem,
\begin{equation}\label{eq1.3}
\begin{aligned}
    (-\Delta)^s  u&= f(x,u)   \text{ in } \,\,\Omega,\\
      u &=0 \,\text{in} \,~ \mathbb R^N \setminus \Omega,
   \end{aligned}
    \end{equation}
    where $\Omega\subset\mathbb{R}^N$ is a bounded domain with smooth boundary $\partial \Omega$, $0<s<1$, $N>2s$, $f:\bar{\Omega}\cross \mathbb{R}\rightarrow \mathbb{R}$ is continuous. The existence of a positive solution, a negative solution, and a sign-changing solution to the problem \eqref{eq1.3} was guaranteed by Chang and Wang \cite{CW2014}. They employed the method of invariant sets of descending flow combined with the Caffarelli and Silvestre \cite{CS2007} extension and an equivalent definition of nonlocal to local operator introduced by Br\"{a}ndle {\it{et al.}} \cite{BCDS2013}. Moreover, they proved that the sign-changing solution has exactly two nodal domains. Gu {\it{et al.}}\cite{GYZ2017} investigated the problem with an integro-differential operator using the constrained variational method and the quantitative deformation lemma. Later, using the method of invariant sets of the descending flow, Deng and Shuai \cite{DS2018} proved the existence of a positive solution, a negative solution, and a sign-changing solution to the problem \eqref{eq1.3} under some suitable conditions. In particular, they obtained that the least energy of the sign-changing solutions is strictly greater than the ground state energy when $f$ satisfies a monotonicity condition. Li $et~al.$ \cite{LST2017} proved infinitely many sign-changing solutions for the B\'rezis-Nirenberg problem when $f(x,u)=|u|^{2^*_s-2}u+\lambda u$ using the minimax method and invariant sets of the descending flow.
    
   Recently, Chang {\it{et al.}} \cite{CNW2019} consider the following problem,
   \begin{equation}\label{eq1.4}
\begin{aligned}
    (-\Delta)^s_p  u&= f(x,u)   \text{ in } \,\,\Omega,\\
      u &=0 \,\text{in} \,~ \mathbb R^N \setminus \Omega,
   \end{aligned}
    \end{equation}
where $\Omega\subset\mathbb{R}^N$ is a bounded domain with smooth boundary $\partial \Omega$, $0<s<1<p<\infty$, $N>sp$, $f:\bar{\Omega}\cross \mathbb{R}\rightarrow \mathbb{R}$ is continuous. In \cite{CNW2019}, the authors established the existence and multiplicity of a sign-changing solution to the problem \eqref{eq1.4} by applying the method of invariant sets of descending flow. Moreover, they employed the Nehari manifold method combined with a constrained variational technique and Brouwer degree theory to guarantee the existence of a least energy sign-changing solution whose energy is strictly greater than twice that of the ground state energy. Frassu and Iannizzotto \cite{FI2021} proved the existence of the smallest positive, the biggest negative, and a sign-changing solution to problem \eqref{eq1.4} using the Fučik spectrum and a truncation technique. In \cite{GSKC2019}, Ghosh $et~al.$ established the existence of the least energy of the sign-changing solutions for the singular problem employing a cut-off technique and Nehari manifold method. For further studies in this direction, we refer to \cite{NPV2012,RS2014,S2007,LST2017,DS2018,GSKC2019,FI2021} and the references therein. 

\par Recently, elliptic PDEs with mixed local and non-local operators have attracted significant interest from researchers considering its importance in theoretical developments as well as its real-world applications in population dynamics \cite{DV2021}, Brownian motion and L\'evy process \cite{DPV2023}. Following the work due to Dipierro \textit{et al}.  \cite{DV2021}, significant contributions are made in the context of existence and regularity of solutions. For instance, in \cite{BMV2024}, Biagi \textit{et al}. obtained the necessary and sufficient condition for the existence and uniqueness of a positive weak solution to the Brezis-Oswald type problem, which is given by
\begin{equation}\label{1.6}
\begin{aligned}
    -\Delta_p u +(-\Delta)^s_p  u&= g(x,u)   \text{ in } \,\,\Omega,\\
   % u&>0  \text{ in } \,\,\Omega,\\
    u &=0 \,\text{in} \,~ \mathbb R^N \setminus \Omega,
   \end{aligned}
    \end{equation}
    where $\Omega\subset\mathbb{R}^N$ is a bounded domain with smooth boundary $\partial \Omega$, $1<p<\infty$, $0<s<1$, $N>p$, the nonlinearity $g:\bar{\Omega}\cross \mathbb{R}\rightarrow \mathbb{R}$ is satisfies certain growth conditions. Da Silva and Salort \cite{DS2020} proved the existence of at least one positive solution to the problem \eqref{1.6} involving concave-convex nonlinearities. Moreover, they investigated the asymptotic behavior of weak solutions as $p \rightarrow \infty$. In Da Silva $et ~al.$ \cite{DFV2024}, the authors established the existence and multiplicity of solutions to the problem \eqref{1.6} with both $p$-sublinear and $p$-superlinear growth using the Krasnoselskii's genus and the Lusternik-Schnirelman category theory. For the existence and regularity of positive solutions to \eqref{1.6} for $p=2$, we refer to Biagi $et~al.$ \cite{BDVV2022,BVDV2021}, Dipierro $et.al$ \cite{DV2021,DPV2023}, and Su $et.al$ \cite{SVWZ2025}. For further study in this direction, we refer to \cite{BG2024,LGG2024,BMV2024,BDVV2022,BVDV2021,BV2024,BV2024N,DM2022,GK2022,MMV2023,DFV2024} and references therein. 
    \par Recently, Su, Valdinoci, Wei and Zhang \cite{SVWZ2024} investigated the sign-changing solutions to the following problem,
\begin{equation}\label{eq1.5}
\begin{aligned}
    -\Delta u +(-\Delta)^s u&= \lambda |u|^{q-2}u+g(x,u)   \text{ in } \,\,\Omega,\\
      u &=0 \,\text{in} \,~ \mathbb R^N \setminus \Omega,
   \end{aligned}
    \end{equation}
    where $\Omega\subset\mathbb{R}^N$ is a bounded domain, $1<q<2$, $\lambda>0$, $g(x,u)$ satisfies some conditions. In \cite{SVWZ2024}, the authors proved that the problem \eqref{eq1.5} possesses a minimum of five nontrivial weak solutions. If $\Omega$ has $C^{1,1}$ boundary, then using descending flow in ordered spaces and the Hopf-type Lemma the authors proved that there exists $\lambda_0$ such that for $\lambda \in (0, \lambda_0)$, problem \eqref{eq1.5} had at least six nontrivial classical solutions, including two sign-changing solutions. In addition, employing the Nehari manifold method, they obtained that there exists a $\lambda^*>0$ such that for $\lambda\in(0,\lambda^*)$, problem \eqref{eq1.5} has at least six nontrivial classical solutions, including one sign-changing solution whenver $\Omega$ is of class $C^{1,1}$. To the best of our knowledge, the study due to Su $et ~al.$ \cite{SVWZ2024} is the only result available in the literature for sign-changing solutions. 
   % when $g(x,u)=\lambda|u|^{q-2}u+|u|^{p^*-2}u$.
    \par Motivated by the above-mentioned studies, we consider the following problem involving the mixed local and nonlocal $p$-Laplacian operator
    \begin{equation}\label{eq1.6}
\begin{aligned}
    -\Delta_p u +(-\Delta)^s_p  u&= f(x,u)   \text{ in } \,\,\Omega,\\
      u &=0 \,\text{in} \,~ \mathbb R^N \setminus \Omega,
   \end{aligned}
    \end{equation}
where $\Omega\subset\mathbb{R}^N$ is a bounded domain with smooth boundary $\partial \Omega$, $1<p<\infty$, $0<s<1$, $N>p$, $f:\bar{\Omega}\cross \mathbb{R}\rightarrow \mathbb{R}$ is continuous.
   Before we proceed further, we present the following table of references that inspired the consideration of our problem. 
\begin{table}[H]
    \centering
\begin{tabular}{|c|c|c|c|c|c|c|}
 \hline
 Operators & $-\Delta$ &  $-\Delta_p$ &  $(-\Delta)^s$& $(-\Delta)^s_p$ & $-\Delta+(-\Delta)^s$& $-\Delta_p+(-\Delta)^s_p$ \\
  & & & & & & \\
  [0.90ex]
  \hline
 Sign-changing &\, \cite{BW1996,BCW2000,BW2003,CCN1997} &\,\cite{BL2004,BLW2005}  &\,\cite{CW2014,DS2018}  &\, \cite{CNW2019,FI2021} &\,\cite{SVWZ2024} &\,-- \\
solution &\cite{LW2004,LS2001,W1991} & & \cite{GYZ2017,LST2017}&\cite{GSKC2019} & & \\
 & & & & & & \\
 [0.90ex]
 \hline
\end{tabular}
\caption{}\label{table:1}
\end{table}

\par It is noteworthy to mention here that due the nonlinearity of the nonlocal operator, we cannot have the decomposition $\Phi(u)=\Phi(u^+)+\Phi(u^-)$ and $\langle \Phi'(u),u\rangle=\langle \Phi'(u^+),u^+\rangle+\langle \Phi'(u^-),u^-\rangle$ for $u=u^++u^-$, where $\Phi$ is the functional corresponding to \eqref{MP}. Moreover, when $p<2$, the energy is restricted to be of class $C^1$ and thus one needs to construct an appropriate pseudo-gradient vector field. Therefore, we first establish the necessary results and apply a critical point theorem combined with the method of invariant sets of descending flow developed by Liu $et~al.$ \cite{LLW2015} to guarantee the existence and multiplicity of sign-changing solutions to the problem \eqref{MP}. We now state our first main result.
\begin{theorem}\label{T2.3}
    Assume that conditions $(f_1), (f_2)$ and $(f_3)$ hold. Then the problem \eqref{MP} has a sign-changing solution. Furthermore, if $f$ satisfies the condition $(f_6)$, then the problem \eqref{MP} possesses infinitely many sign-changing solutions.
\end{theorem}

\par The conditions $(f_3)$ and $(f_4)$ were introduced by Ambrosetti and Rabinowitz \cite{AR1973}, which are now known as Ambrosetti-Rabinowitz (AR)-conditions. It is evident that $(f_4)$ constitutes a weaker condition than $(f_3)$. Weth \cite{W2006}, introduced the concept of ``doubling energy" for sign-changing solutions, that is, the energy level for the least energy of sign-changing solutions is strictly greater than twice that of the ``ground-state" energy. In the next theorem, we employ the Nehari manifold method combined with the Brouwer degree theory and a constrained variational argument with the weak condition $(f_4)$, to establish the existence of a least energy sign-changing solution and a nontrivial ground-state solution to the problem \eqref{MP}.  Moreover, we extend the doubling property for the mixed local and nonlocal $p$-Laplacian. 
 
 \begin{theorem}\label{T2.4}
     Assume that $f\in C^1(\bar{\Omega}\cross \mathbb{R},\mathbb{R})$ and the conditions $(f_1)$, $(f_2)$, $(f_4)$, $(f_5)$ hold. Then problem \eqref{MP} admits one least energy sign-changing solution $u^*\in \mathbb{X}_0^{s,p}(\Omega)$ and one nontrivial solution $u_*\in \mathbb{X}_0^{s,p}(\Omega)$ such that $m_s=\Phi(u^*)$ and $c_s=\Phi(u_*)$ such that $m_s>2c_s$. (see Section \ref{sec2}).
 \end{theorem}
 \begin{cor} 
       The results in Theorem \ref{T2.4} true when $(f_3)$ is used instead of $(f_4)$.
 \end{cor}
\begin{remark}
    We point out that due to the presence of nonlocal terms, we fail to conclude that the least energy sign-changing solution obtained in Theorem \ref{T2.4}, has exactly two nodal domains even for $p=2$, (see Gu {\it{et al.}}\cite{GYZ2017}, Teng \textit{et al.} \cite{teng2015}).
\end{remark}

   \par To the best of our knowledge, Theorem \ref{T2.3} and Theorem \ref{T2.4} for sign-changing solutions of a mixed local and nonlocal $p$-Laplacian are new and possibly the first in the literature even for $p=2$.

\par The rest of the paper is organized as follows: In Section \ref{sec2}, we recall some fundamental results and develop the necessary tools for the solution space related to our problem. In Section \ref{sec3}, we review some important results that are applicable to our problem. Section \ref{sec4} is devoted to establishing the existence of sign-changing solutions to problem \eqref{MP} and the existence of infinitely many such solutions. Finally, in Section \ref{sec5}, we derive the existence of least energy of the sign-changing solutions and the existence of ground-state solutions with the doubling energy property.

\section{Preliminaries and notion of solutions}\label{sec2}
\noindent In this section, we recall some fundamental properties of Sobolev spaces and define the notion of sign-changing solutions. Unless specified, throughout the paper, we assume that $\Omega\subset\mathbb{R}^N$ is a bounded domain with smooth boundary $\partial\Omega$. Recall the definitions of Sobolev spaces and the fractional Sobolev spaces \cite{NPV2012}. For $1\leq p<\infty$, the Sobolev space $W^{1,p}(\Omega)$ is defined as 
   \begin{align*}
       W^{1,p}(\Omega)=\{u \in L^p(\Omega):\nabla u \in L^p(\Omega)\},
       \end{align*}
       which is a Banach space equipped with the norm 
       \begin{equation}\label{n1-1}
           \|u\|_{W^{1,p}(\Omega)}=\|u\|_{L^p(\Omega)}+ \|\nabla u\|_{L^p(\Omega)}. 
       \end{equation}
For every $0<s<1\leq p<\infty$, the fractional Sobolev space $W^{s,p} (\Omega)$ is defined as
       \begin{equation*}
       \begin{aligned}
           W^{s,p}(\Omega)&=\left\{u \in L^p(\Omega): \frac{|u(x)-u(y)|}{|x-y|^{\frac{N}{p}+s}} \in L^p(\Omega \times \Omega)\right\},
            \end{aligned}
           \end{equation*}
           which is a Banach space endowed with the norm
           \begin{equation}\label{n s}
               \|u\|_{W^{s,p}(\Omega)}=\left(\|u\|^p_{L^p(\Omega)}+ [u]_{W^{s,p}(\Omega)}^p\right)^{\frac{1}{p}},
           \end{equation}
           where $[u]_{W^{s,p}(\Omega)}$ is the Gagliardo seminorm of $u$ which is given by
           \begin{equation}\label{n g}
               [u]_{W^{s,p}(\Omega)}=\bigg(\int_\Omega \int_\Omega \frac{|u(x)-u(y)|^p}{|x-y|^{N+ps}}\d x \d y\bigg)^\frac{1}{p}.
           \end{equation}
    Note that the norm \eqref{n s} is equivalent to the following norm,
          \begin{align}\label{n s1}
              \|u\|_{W^{s,p}(\Omega)}=\|u\|_{L^p(\Omega)}+[u]_{W^{s,p}(\Omega)}
          \end{align}
The spaces $W_0^{1,p}(\Omega)$ and $W_0^{s,p}(\Omega)$ are defined as the closure of $C_c^\infty(\Omega)$ in $W^{1,p}(\Omega)$ and $W^{s,p}(\Omega)$ with respect to the norm in \eqref{n1-1} and \eqref{n s}, respectively. Moreover, the Sobolev spaces $W_0^{1,p}(\Omega)$ and $W_0^{s,p}(\Omega)$ are characterized as
    \begin{align*}
        W_0^{1,p}(\Omega)&=\{u\in W^{1,p}(\Omega):u=0~\text{on}~ \partial \Omega\}~\text{and}\\
        W_0^{s,p}(\Omega)&=\{u\in W^{s,p}(\Omega):u=0~\text{in}~ \mathbb R^N \setminus \Omega\}.
    \end{align*}
    Let $\Omega$ be bounded. On using the Poincar\'e inequalities, we conclude that the norm \eqref{n1-1} reduces to the following homogeneous norm in $W^{1,p}_0(\Omega)$,
        \begin{equation}\label{nh1}
            \|u\|_{W^{1,p}_0(\Omega)}=\left(\int_\Omega |\nabla u|^p\right)^\frac{1}{p}.
        \end{equation}
 Similarly, the seminorm \eqref{n g} serves as a norm on $W^{s,p}_0(\Omega)$, that is $\|u\|_{W_0^{s,p}(\Omega)}=[u]_{W^{s,p}(\Omega)}$.

Recall the Sobolev inequality \cite{EVANS2022}. For every $u \in W^{1,p}(\mathbb{R}^N)$ with $1\leq p<N$, we have 
   \begin{equation}\label{eq2.9}
           \bigg(\int_{\mathbb{R}^N}|u|^{p^*}dx\bigg)^\frac{1}{p^*}\leq C\bigg(\int_{\mathbb{R}^N}|\nabla u|^{p}dx\bigg)^\frac{1}{p},
       \end{equation}
       where $C>0$ is the best embedding constant and $p^*=\frac{Np}{N-p}$. We have the fractional Sobolev inequality \cite[Theorem 6.5]{NPV2012}. For every $u \in W^{s,p}(\mathbb{R}^N)$ with $1\leq p<\frac{N}{s}$, we have 
   \begin{equation}
           \bigg(\int_{\mathbb{R}^N}|u|^{p_s^*}dx\bigg)^\frac{1}{p_s^*}\leq C\bigg(\int_{\mathbb{R}^N} \int_{\mathbb{R}^N} \frac{|u(x)-u(y)|^p}{|x-y|^{N+sp}}\d x\d y \bigg)^\frac{1}{p},
       \end{equation}
        where $C>0$ is the best embedding constant and $p_s^*=\frac{Np}{N-ps}$. We now state the following embedding results for Sobolev spaces \cite{EVANS2022,NPV2012}. 
 \begin{lemma}\label{sob eql} 
Let $\Omega$ be a bounded domain with Lipschitz boundary $\partial\Omega$. Then 
\begin{itemize}
    \item[$(a)$] For $1\leq p<N$, the spaces $W^{1,p}(\Omega)$ and $W_0^{1,p}(\Omega)$ are continuously embedded in $L^q(\Omega)$ for all $q\in[1,p^*]$ and the embedding is compact for $1\leq q<p^*=\frac{Np}{N-p}$.
    \item[$(b)$] For $0<s<1$ with $1\leq p<\frac{N}{s}$, the spaces $W^{s,p}(\Omega)$ and $W_0^{s,p}(\Omega)$ are continuously embedded in $L^q(\Omega)$, $\forall\,q\in[1,p_s^*]$ and the embedding is compact for $1\leq q<p_s^*=\frac{Np}{N-ps}$.
\end{itemize}
 \end{lemma}
  The next lemma is due to \cite[Lemma 2.1]{BSM2022} and \cite[Proposition 2.2]{NPV2012}, which plays a crucial role in studying our problem \eqref{MP}.
   \begin{lemma}
   Let $\Omega\subset\mathbb{R}^N$ be a bounded domain with Lipschitz boundary $\partial\Omega$. Then for $0<s<1 \leq p< \infty$, there exists $C=C(N,p,s)>0$ such that 
       \begin{equation}\label{eq2.5IM}
       \begin{aligned}
    \|u\|_{W^{s,p}(\Omega)} \leq C\|u\|_{W^{1,p}(\Omega)}, \,\,\, \forall\, u\in W^{1,p}(\Omega).
     \end{aligned}
        \end{equation}
Moreover, for every $u \in W_0^{1,p}(\Omega)$ with $u=0$ in $\mathbb R^N\setminus\Omega$, we have
        \begin{equation}\label{sg-emb}
           \int_{\mathbb R^N} \int_{\mathbb R^N} \frac{|u(x)-u(y)|^p}{|x-y|^{N+sp}} \d x \d y \leq C \int_\Omega |\nabla u|^p \d x. 
       \end{equation}
       In particular, we have
        \begin{equation}\label{sg-emb2}
           \int_{\Omega} \int_{\Omega} \frac{|u(x)-u(y)|^p}{|x-y|^{N+sp}} \d x \d y \leq C \int_\Omega |\nabla u|^p \d x \,\,\, \forall \, u \in W_0^{1,p}(\Omega).
           %~\text{with}~u=0~\text{in}~\mathbb R^N\setminus\Omega.
       \end{equation}
       \end{lemma}
 \noindent With the preliminaries above, we now define the solution space for our problem \eqref{MP}.   
\begin{definition}
  Let $\Omega\subset\mathbb{R}^N$ be a bounded domain with Lipschitz boundary $\partial \Omega$ and let $0<s<1\leq p<\infty$. We define the Sobolev space $\mathbb{X}_0^{s,p}(\Omega)$ as the closure of $C_c^\infty(\Omega)$ with respect to the following norm,
  \begin{align}\label{eq2.7 N}
      \|u\|_{\mathbb{X}_0^{s,p}(\Omega)}&=\bigg(\int_\Omega|\nabla u|^p \d x+{\int_{\mathbb{R}^N} \int_{\mathbb{R}^N}} \frac{|u(x)-u(y)|^p}{|x-y|^{N+sp}}\d x\d y \bigg)^\frac{1}{p},\,\,\,\,\forall\, u\in C_c^\infty(\Omega).
      \end{align} 
      \end{definition}

      \begin{remark}\label{rm 2.1}
      Recall the Poincar\'e inequality for $W_0^{1,p}(\Omega)$. For all $ u \in W_0^{1,p}(\Omega)$, there exists $C>0$ such that
  \begin{align}\label{lmn2.1 PQ}
        \|u\|_{L^p(\Omega)}\leq C\|\nabla u\|_{L^p(\Omega)}.
        \end{align}
        Thus, using \eqref{sg-emb}, \eqref{sg-emb2} and \eqref{lmn2.1 PQ}, we obtain the following equivalent norms on $\mathbb{X}_0^{s,p}(\Omega)$.
      \begin{align*}
          \|u\|_{\mathbb{X}_0^{s,p}(\Omega)}&:=\bigg(\int_\Omega|\nabla u|^p \d x \bigg)^\frac{1}{p},\,\,\,\forall \, u \in \mathbb{X}_0^{s,p}(\Omega) \text{ and }\\
          \|u\|_{\mathbb{X}_0^{s,p}(\Omega)}&:=\bigg(\int_\Omega|\nabla u|^p \d x+{\int_{\Omega} \int_{\Omega}} \frac{|u(x)-u(y)|^p}{|x-y|^{N+sp}}\d x\d y \bigg)^\frac{1}{p},\,\,\,\forall \, u \in \mathbb{X}_0^{s,p}(\Omega).
      \end{align*}
Therefore, the space $\mathbb{X}_0^{s,p}(\Omega)$ can be characterized as 
    {\begin{equation}\nonumber 
          \mathbb{X}_0^{s,p} (\Omega)={\{u\in W_0^{1,p} (\Omega): u=0 \text{ in } \mathbb R^N \setminus \Omega}  \}.
          \end{equation}}
    \end{remark}
    \noindent Note that on using \eqref{eq2.9}, we get the mixed Sobolev inequality on $\mathbb{X}_0^{s,p} (\mathbb{R}^N)$, which is given by 
     \begin{equation}\label{eq 2.15}
           \bigg(\int_{\mathbb{R}^N}|u(x)|^{p^*}\d x\bigg)^\frac{1}{p^*}\leq C\bigg(\int_{\mathbb{R}^N}|\nabla u(x)|^{p}\d x +\int_{\mathbb{R}^N} \int_{\mathbb{R}^N} \frac{|u(x)-u(y)|^p}{|x-y|^{N+sp}}\d x\d y\bigg)^\frac{1}{p},
       \end{equation}
       where $C>0$ is the best embedding constant and $p^*=\frac{Np}{N-p}>\frac{Np}{N-ps}:=p_s^*$. Therefore,  Remark \ref{rm 2.1} combined with the inequality \eqref{eq2.9} assert that
          \begin{equation*}
            \|u\|_{L^p{^*}(\Omega)}=\|u\|_{L^p{^*}(\mathbb R^N)} \leq C\|\nabla u\|_{L^p(\mathbb R^N) }\leq C\|u\|_{\mathbb{X}_0^{s,p}(\Omega)}, \, \forall\, u \in \mathbb{X}_0^{s,p}(\Omega).
             \end{equation*} 
    \begin{theorem}\label{thm cpt}
     Let $0<s<1\leq p<\infty$ and let $\Omega\subset\mathbb{R}^N$ be a bounded domain with Lipschitz boundary $\partial \Omega$. Then we have
             \begin{equation}\label{eq2.8S}
                  \|u\|_{L^p{^*}(\Omega)} \leq C\|u\|_{\mathbb{X}_0^{s,p}(\Omega)},\, \forall\, u \in \mathbb{X}_0^{s,p}(\Omega).
             \end{equation}
             Moreover, the embedding $\mathbb{X}_0^{s,p}(\Omega)\hookrightarrow L^q(\Omega)$ is continuous for $1\leq q\leq p^*$ and is compact $1\leq q< p^*$.
    \end{theorem}
  \noindent The following theorem characterizes the space $\mathbb{X}^{s,p}_0(\Omega)$.
            \begin{theorem}\label{thm prop}
                Let $0<s<1\leq p<\infty$ and let $\Omega\subset\mathbb{R}^N$ be a bounded domain with Lipschitz boundary $\partial \Omega$. Then, the space $\mathbb{X}^{s,p}_0(\Omega)$ is a Banach space endowed with the norm \eqref{eq2.7 N}, for all $p\in[1,\infty)$. Moreover, it is separable for all $p\in[1,\infty)$ and is reflexive for all $p\in(1,\infty)$. In particular, when $p=2$, the space $\mathbb{X}_0^{s,2}(\Omega)$ reduces to a Hilbert space with respect to the inner product, 
                \begin{equation}
                    \langle u,v\rangle_{\mathbb{X}_0^{s,2}(\Omega)}=\int_{\Omega} \nabla u\cdot\nabla v \d x+{{\int_{\mathbb{R}^N} \int_{\mathbb{R}^N}}}\frac{(u(x)-u(y))(v(x)-v(y))}{|x-y|^{N+2s}} \d x \d y,
                \end{equation}
                where ``$\cdot$" denotes the standard scalar product in $\mathbb{R}^N$.
            \end{theorem}
            \begin{proof}
                By definition, $\mathbb{X}^{s,p}_0(\Omega)$ is a Banach space. For any $u\in \mathbb{X}^{s,p}_0(\Omega)$, choose $A_u(x)=\nabla u(x)$ and $B_u(x,y)=\frac{u(x)-u(y)}{|x-y|^{\frac{N}{p}+s}}$ and define the map $T: \mathbb{X}^{s,p}_0(\Omega) \rightarrow L^p(\Omega)\cross L^p({\mathbb{R}^N \cross \mathbb{R}^N})$ such that
                $$T(u)=(A_u,B_u).$$
             Since, $\|Tu\|_{L^p(\Omega)\cross L^p({\mathbb{R}^N \cross \mathbb{R}^N})}=\|u\|_{\mathbb{X}^{s,p}_0(\Omega)}$  $\forall\,u\in \mathbb{X}^{s,p}_0(\Omega)$, we obtain $T$ is an isometry into the closed subspace of $L^p(\Omega)\cross L^p({\mathbb{R}^N \cross \mathbb{R}^N}))$. Thus we get $\mathbb{X}^{s,p}_0(\Omega)$ is is reflexive for all $p\in(1,\infty)$ and is separable for all $p\in[1,\infty)$. Finally, using the fact that $W_0^{1,2}(\Omega)$ and $W_0^{s,2}(\Omega)$ are Hilbert spaces, we conclude $\mathbb{X}^{s,2}_0(\Omega)$ is a Hilbert space.
            \end{proof}
\begin{definition}
  We say $u\in \mathbb{X}_0^{s,p}(\Omega)$ is a weak solution to the problem \eqref{MP} if
    \begin{align}\label{WFMP}
    \int_\Omega |\nabla u|^{p-2} \nabla u\cdot\nabla \phi \d x&+{\int_{\mathbb{R}^N} \int_{\mathbb{R}^N}}\frac{|u(x)-u(y)|^{p-2}(u(x)-u(y))(\phi(x)-\phi(y))}{|x-y|^{N+sp}} \d x \d y \nonumber \\
    &= \int_\Omega f(x,u)\phi \d x,~\forall\,\phi\in \mathbb{X}_0^{s,p}(\Omega).
\end{align}
\end{definition}
\noindent Define the energy functional $\Phi:\mathbb{X}_0^{s,p}(\Omega)\rightarrow \mathbb R$ as follow:
\begin{align}\label{1EF}
    \Phi(u)=\frac{1}{p} \int_\Omega|\nabla u|^p \d x+\frac{1}{p} {\int_{\mathbb{R}^N} \int_{\mathbb{R}^N}}\frac{|u(x)-u(y)|^p}{|x-y|^{N+sp}}\d x\d y -\int_\Omega F(x,u) \d x.
\end{align}
From \cite{R1986} combined with $(f_1)$ and $(f_2)$, we conclude that $\Phi$ is $C^1$. Hence, 
\begin{align}\label{eq2.20}
     \langle\Phi'(u),\phi\rangle = &  {\int_{\mathbb{R}^N} \int_{\mathbb{R}^N}} \frac{|u(x)-u(y)|^{p-2}(u(x)-u(y))(\phi(x)-\phi(y))}{|x-y|^{N+sp}} \d x \d y  \nonumber\\
     &+ \int_\Omega |\nabla u|^{p-2} \nabla u\cdot\nabla \phi \d x - \int_\Omega f(x,u)\phi \d x,~\forall\,\phi \in \mathbb{X}_0^{s,p}(\Omega),
\end{align}
where $\langle.,.\rangle:=\langle.,.\rangle_{{\mathbb{X}_0^{s,p}(\Omega)}^*,\mathbb{X}_0^{s,p}(\Omega)}$ denotes the dual pair and ${\mathbb{X}_0^{s,p}(\Omega)}^*$ is the dual space of ${\mathbb{X}_0^{s,p}(\Omega)}$. Clearly, the critical points of the energy functional $\Phi$ are weak solutions to the problem \eqref{MP}.

\par Putting $\phi=u$ in \eqref{eq2.20}, we get
\begin{align}
    \langle\Phi'(u),u\rangle=\int_\Omega |\nabla u|^{p}  \d x+{\int_{\mathbb{R}^N} \int_{\mathbb{R}^N}} \frac{|u(x)-u(y)|^{p}}{|x-y|^{N+sp}} \d x \d y - \int_\Omega f(x,u)u \d x,
\end{align}
for all $u \in \mathbb{X}_0^{s,p}(\Omega).$ We now define the Nehari manifold $\mathcal{N}$ and the set of sign-changing solutions $\mathcal{M}$ as follow:
$$\mathcal{N}=\{u\in \mathbb{X}_0^{s,p}(\Omega)\setminus \{0\} :\langle\Phi'(u),u\rangle=0\}$$ and
$$\mathcal{M}=\{u\in \mathbb{X}_0^{s,p}(\Omega):u^{\pm}\neq0, \langle\Phi'(u),u^+\rangle=\langle\Phi'(u),u^-\rangle=0\},$$
where $$u^+=max\{u(x),0\}=\frac{u+|u|}{2} \text{ and } u^-=min\{u(x),0\}=\frac{u-|u|}{2}.$$
 We set 
$$m_s=\inf_{u\in \mathcal{M}} \Phi(u)~\text{and}~c_s=\inf_{u\in \mathcal{N}} \Phi(u).$$
Note that $\mathcal{N}$ contains all the nontrivial (ground-state) solutions and $\mathcal{M}$ contains all the sign-changing solutions to the problem \eqref{MP}. Moreover, $ \mathcal{M}\subset \mathcal{N}$. Throughout the paper, we denote $C$ as a positive constant whose value may vary even in the same line. 

\section{Some Important Results}\label{sec3}
In this section, we review two essential critical point theorems\cite{LLW2015} in a metric space, which are useful for studying our problem in the subsequent sections of this paper. Let $(X,d)$ be a complete metric space with $Y_1, Y_2\subset X$ being open sets. Let $\Phi\in C^1(X,\mathbb R)$ and $a,b,c\in \mathbb R$. Before presenting the results, it is important to define the following notations. We denote $Z=Y_1 \cap Y_2$ and $W=Y_1\cup Y_2$ with $\sum=\partial Y_1 \cap \partial Y_2$. Moreover, $M, M_c, M([a,b])$ and $\Phi^c$ are defined as $M=\{u\in X: \Phi'(u)=0\},$ $M_c=\{u\in X:\Phi(u)=c, \Phi'(u)=0\},$ $M([a,b])=\{u\in X:a\leq \Phi(u)\leq b, \Phi'(u)=0\}$ and $\Phi^c=\{u\in X:\Phi(u)\leq c\},$ respectively. 
%$\Phi^{-1}([a,b])=\{u\in X :a\leq \Phi(u)\leq b\}$,
We now state the following definitions and theorems from \cite{LLW2015}.
\begin{definition}
    The set $\{Y_1, Y_2\}$ is called an admissible family of invariant sets with respect to $\Phi$ at the level $c$, if it satisfies the following criterion: If $M_c\setminus W=\emptyset$, then there exists $\epsilon_0$ such that, for every $\epsilon\in(0,\epsilon_0)$, there exists a continuous mapping $\sigma:X\rightarrow X$ satisfying
    \begin{itemize}
        \item[$(a)$] $\sigma(\overline{Y_1})\subset \overline{Y_1}$ and $\sigma(\overline{Y_2})\subset \overline{Y_2}$
        \item[$(b)$] $\sigma|_{\Phi^{c-\epsilon}}=I$, $I$ is the identity map,
        \item[$(c)$] $\sigma(\Phi^{c+\epsilon}\setminus W)\subset\Phi^{c-\epsilon}$.
    \end{itemize}
    \end{definition}
      \begin{definition}
We say $G: X\rightarrow X$ is an isometric involution if $G^2=I$ and $d(Gx, Gy)=d(x,y)$ for all $x,y\in ~X$, where $I$ denotes the identity mapping. Moreover, we say $E\subset X$ is symmetric if $Gu\in E$ for all $u\in E.$
     \end{definition}
    \begin{definition}
        Let $\Gamma=\{A\subset X: A~\text{is closed, symmetric and}~0\notin A\}$. The genus of $E\in\Gamma$, denoted by $\gamma (E)$, is the smallest positive integer $n$ such that there exists an odd and continuous map $h:E\rightarrow \mathbb{R}^n\setminus {0}$. If such mapping does not exist, then $\gamma(E)=\infty$. We denote $\gamma(\emptyset)=0$.
    \end{definition}
    \begin{definition}
    Let $(X,d)$ be a complete metric space and $\Phi\in C^1(X,\mathbb R)$.
       The set $\{Y_1,Y_2\}$ is called the $G$-admissible family of invariant sets with respect to $\Phi$ at level $c$, if it fulfills the following criterion: There exists $\epsilon_0>0$ and a symmetric neighborhood $N_c$ of $M_c\setminus W$ with $\gamma(\overline{N}_c)<+\infty$, $(N_c=\emptyset~\text{if}~M_c\setminus W=\emptyset)$ such that, for every $\epsilon\in(0,\epsilon_0)$, there exists a continuous mapping $\sigma:X\rightarrow X$ satisfying 
       \begin{itemize}
         \item[$(a)$] $\sigma(\overline{Y_1})\subset \overline{Y_1}$ and $\sigma(\overline{Y_2})\subset \overline{Y_2}$,
         \item[$(b)$] $\sigma \circ G=G \circ \sigma$,
        \item[$(c)$] $\sigma|_{\Phi^{c-2\epsilon}}=I$, $I$ is the identity mapping,
        \item[$(d)$] $\sigma({\Phi^{c+\epsilon}}\setminus(N_c\cup W))\subset\Phi^{c-\epsilon}$.
        \end{itemize}
        \end{definition}
        \begin{theorem}\label{T3.1}[Theorem 2.4,\cite{LLW2015}]
        Let $\Phi\in C^1(X,\mathbb{R})$, $Y_1$ and $Y_2$ be open subsets of $X$. Let $\{Y_1,Y_2\}$ be an admissible family of invariant sets with respect to $\Phi$ at level $c\geq c_*=\inf_{u\in \sum}\Phi(u)$ and there exists a continuous mapping $\psi:\Delta \rightarrow X$ such that 
        \begin{itemize}
            \item[$(a)$] $\psi(\partial_1\Delta)\subset Y_1$ and $\psi(\partial_2\Delta)\subset Y_2$,
            \item[$(b)$] $\psi(\partial_0\Delta)\cap Z=\emptyset$,
            \item[$(c)$] $\sup\limits_{u\in \psi (\partial_0\Delta)}\Phi(u)<c_*,$
        \end{itemize}
        where 
        \begin{align}
            \Delta=\{(t_1,t_2)\in\mathbb{R}^2:t_1,t_2\geq 0, t_1+t_2\leq 1\}, \nonumber\\
            \partial_0 \Delta=\{(t_1,t_2)\in\mathbb{R}^2:t_1,t_2\geq 0, t_1+t_2= 1\}, \nonumber \\
            \partial_1 \Delta=\{(t_1,t_2)\in\mathbb{R}^2:t_1= 0, 0\leq t_2\leq 1\}=\{0\}\cross\{[0,1]\}, \nonumber\\
            \partial_2 \Delta=\{(t_1,t_2)\in\mathbb{R}^2: 0\leq t_1\leq 1, t_2=0\}=\{[0,1]\} \cross \{0\}.\nonumber
        \end{align}
        For $\Gamma=\{\phi\in C(\Delta,X):\phi(\partial_1\Delta)\subset Y_1, \phi(\partial_2\Delta)\subset Y_2, \phi|_{\partial_0\Delta}=\psi|_{\partial_0\Delta}\}$, define,
        $$c_0=\inf_{\phi\in\Gamma}\sup_{u\in \phi(\Delta)\setminus W}\Phi(u).$$
        Then $c_0\geq c_*$ and $M_{c_0}\setminus W \neq \emptyset.$
\end{theorem}
        \begin{theorem}\label{T3.2}[Theorem 2.5, \cite{LLW2015}]
            Let $\Phi\in C^1(X,\mathbb{R})$ is a G-invariant functional, $Y_1$ and $Y_2$ be open subsets of $X$. Assume that $\{Y_1, Y_2\}$ is G-admissible family of invariant sets with respect to $\Phi$ at level $c\geq c_*:=\inf_{u\in \sum}\Phi(u)$. Suppose for any $n\in \mathbb{N}$, there exists a continuous map $\phi_n:B_{2n}\rightarrow X$ such that 
            \begin{itemize}
            \item[$(a)$] $\phi_n(0)\in Z$, and $\phi_n(-t)=G\phi_n(t),~\forall\,t=(t_1,t_2)\in B_{2n}$, $t_1,t_2\in B_n$,
            \item[$(b)$] $\phi_n(\partial B_{2n})\cap Z=\emptyset$,
            \item[$(c)$] $c_0:=\sup_{u\in F_G\cup \phi_n(\partial B_{2n})}\Phi(u)<c_*=\inf_{u\in \sum}\Phi(u)$, where $B_{2n}=\{t\in \mathbb{R}^{2n}:|t|\leq 1\}$, $F_G=\{u\in X:Gu=u\}$.
        \end{itemize}
             For $j\in \mathbb{N}$, define 
             $$c_j=\inf_{B\in \Gamma_j}\sup_{u\in {B\setminus W}}\Phi(u),$$
             where $\Gamma_j=$ $\{B:B=\phi(B_{2n}\setminus P)$ for some $\phi\in G_n,~n\geq j, ~P\subset B_{2n}$ is open subset with $P=-P$ and $\gamma(\overline{P})\leq n-j\}$ and $G_n=$ $\{\phi:\phi\in C(B_{2n},X),~\phi(-t)=G\phi(t)$ for $t\in B_{2n}$, such that $\phi(0)\in Z$ and $\phi|_{\partial B_{2n}}=\phi_n|_{\partial B_{2n}}\}$. Then, for $j\geq 3$, $c_j\geq c_*$ and $M_{c_j}\setminus W\neq \emptyset$. Moreover, $c_j\rightarrow +\infty$ as $j\rightarrow +\infty.$
        \end{theorem}
    \noindent We conclude this section with the following lemma.
\begin{lemma}\label{inqlmn}[Lemma 3.6, \cite{BL2004}]
    For any $x_1$ and $x_2$ $\in \mathbb{R}^N$, there exist positive constants $d_1$, $d_2$, $d_3$ and $d_4$ such that 
    \begin{align}
        (|x_1|^{p-2}x_1-|x_2|^{p-2}x_2)(x_1-x_2) & \geq \begin{cases}
            & d_1(|x_1|+|x_2|)^{p-2}|x_1-x_2|^2,~\text{if}~p\in(1,2]\\
            & d_3|x_1-x_2|^p,~\text{if}~p\in (2,\infty),
        \end{cases}
    \end{align}
    and 
    \begin{align}
        \left \|x_1|^{p-2}x_1-|x_2|^{p-2}x_2 \right | \leq \begin{cases}
            & d_2|x_1-x_2|^{p-1},~\text{if}~p\in(1,2]\\
            & d_4 (|x_1|+|x_2|)^{p-2}|x_1-x_2|,\text{if}~p\in (2,\infty).
        \end{cases}
    \end{align}
  
\end{lemma}  

 \section{Existence of Sign-Changing solutions}\label{sec4}
 This section is devoted to establishing one of our main theorem, (Theorem \ref{T2.3}). We employ the methods of descending flows to obtain the existence of a sign-changing (or nodal) solution. The existence of solutions is proved by constructing a modified problem. Recall that from $(f_1), (f_2)$ and $(f_3)$, we can deduce that there exists a constant $\beta>0$ such that
$$f(x,u)u+\beta |u|^p>0,~\forall\,x\in \bar{\Omega}~\text{and}~\forall\,u\neq 0.$$
For each $u\in \mathbb{X}_0^{s,p}(\Omega)$, we define a new norm $\|.\|_\beta$, which is equivalent to $\|.\|_{\mathbb{X}_0^{s,p}(\Omega)}$ and is given by
\begin{align}\label{4norm}
    \|u\|_\beta=\bigg(\int_\Omega|\nabla u|^p \d x+{\int_{\mathbb{R}^N} \int_{\mathbb{R}^N}} \frac{|u(x)-u(y)|^p}{|x-y|^{N+sp}}\d x\d y +\beta\int_\Omega |u|^p \d x\bigg)^\frac{1}{p}.
\end{align}
Let us consider the following problem
\begin{align}\label{NP}
    -\Delta_p w+(-\Delta_p)^s w +\beta|w|^{p-2} w &= g(x,u)~\text{in}~\Omega,\\
    w&=0~\text{in}~\mathbb R^N\setminus \Omega,\nonumber
\end{align}
where $g(x,u)=f(x,u)+\beta|u|^{p-2}u$. The weak formulation of the problem \eqref{NP} is 
\begin{align}\label{WFNP}
    \int_\Omega |\nabla w|^{p-2} \nabla w\cdot\nabla \phi \d x+{\int_{\mathbb{R}^N} \int_{\mathbb{R}^N}} \frac{|w(x)-w(y)|^{p-2}(w(x)-w(y))(\phi(x)-\phi(y))}{|x-y|^{N+sp}} \d x \d y \nonumber \\ 
    +\beta \int_\Omega |w|^{p-2} w \phi \d x= \int_\Omega g(x,u)\phi \d x,~\forall\,\phi \in \mathbb{X}_0^{s,p}(\Omega).
\end{align}
We define an operator $A:\mathbb{X}_0^{s,p}(\Omega)\rightarrow \mathbb{X}_0^{s,p}(\Omega)$ as follows. For each $u\in \mathbb{X}_0^{s,p}(\Omega)$, there exists a unique weak solution $w:=A(u)\in \mathbb{X}_0^{s,p}(\Omega)$ to the problem \eqref{NP}. Therefore, weak solutions to the problem \eqref{MP} are fixed points of the operator $A$. We now focus on characterizing the operator $A$. Prior to proceed with the results, let us define the following notations which will be used in the remaining part of the paper.
\begin{itemize}
	\item[$(a)$] $P^+=\{{u\in \mathbb{X}_0^{s,p}(\Omega):u\geq0}\}$ and $P^-=\{{u\in \mathbb{X}_0^{s,p}(\Omega):u\leq 0}\}$
	\item[$(b)$] $P_\epsilon^+=\{u\in \mathbb{X}_0^{s,p}(\Omega):\dist_\beta(u,P^+)<\epsilon\}$ and $P_\epsilon^-=\{u\in \mathbb{X}_0^{s,p}(\Omega):\dist_\beta(u,P^-)<\epsilon\}$,
\end{itemize}
where  $\dist(.,.)$ and $\dist_\beta(.,.)$ denote the distance in $\mathbb{X}_0^{s,p}(\Omega)$ with respect to the norm $\|.\|_{\mathbb{X}_0^{s,p}(\Omega)}$ and $\|.\|_\beta$, respectively. With this, we begin with the following lemma proving that $ A$ is well-defined, continuous and compact.
\begin{lemma}\label{lmn4.1}
    The operator $A$ is well-defined, continuous, and compact.
\end{lemma}
    \begin{proof}
    We divide the proof into three parts.\\
        \textbf{Claim I}: The operator $A$ is well-defined.\\
        {\it Proof of the Claim I:} For any $u\in \mathbb{X}_0^{s,p}(\Omega)$, we define 
        $$J(w)=\frac{1}{p} \int_\Omega|\nabla w|^p \d x+\frac{1}{p} {\int_{\mathbb{R}^N} \int_{\mathbb{R}^N}}\frac{|w(x)-w(y)|^p}{|x-y|^{N+sp}}\d x\d y +\frac{\beta}{p}\int_\Omega|w|^p-\int_\Omega g(x,u)w \d x.$$
   
    Clearly, $J$ is $C^1$ and 
    \begin{align}\label{Der4.4}
        \langle J'(w),\phi\rangle= \int_\Omega |\nabla w|^{p-2} \nabla w \cdot \nabla \phi \d x
        +\beta \int_\Omega |w|^{p-2} w \phi \d x - \int_\Omega g(x,u)\phi \d x \nonumber \\ 
    +{\int_{\mathbb{R}^N} \int_{\mathbb{R}^N}} \frac{|w(x)-w(y)|^{p-2}(w(x)-w(y))(\phi(x)-\phi(y))}{|x-y|^{N+sp}} \d x \d y ,~\forall\,\phi \in \mathbb{X}_0^{s,p}(\Omega).
    \end{align}
    Moreover, $J$ is weakly lower semi-continuous, bounded from below and coercive. Therefore, there exists a $w^*\in \mathbb{X}_0^{s,p}(\Omega)$ which is a minimizer of $J$. Next, we prove that the minimizer $w^*$ of $J$ is unique. Let $w_1$  and $w_2$ be two minimizers of $J$. Setting $\bar{w}_i(x,y)=w_i(x)-w_i(y)$, for $i=1,2$, and taking $\phi(x)={w}_1(x)-{w}_2(x)$ in \eqref{Der4.4}, we obtain
     \begin{align}\label{eq4.5}
     0 =&\int_\Omega (|\nabla w_1|^{p-2} \nabla w_1-|\nabla w_2|^{p-2} \nabla w_2)\cdot\nabla ({w}_1(x)-{w}_2(x))\d x \nonumber\\
     +&{\int_{\mathbb{R}^N} \int_{\mathbb{R}^N}} \left[ \frac{|\bar{w}_1(x,y)|^{p-2}\bar{w}_1(x,y)}{|x-y|^{N+sp}} - \frac{|\bar{w}_2(x,y)|^{p-2}\bar{w}_2(x,y)}{|x-y|^{N+sp}}\right](\bar{w}_1(x,y)-\bar{w}_2(x,y))\d x \d y\nonumber \\ 
    +&\beta \int_\Omega (|w_1|^{p-2} w_1-|w_2|^{p-2} w_2)(w_1-w_2)\d x.
     \end{align}
On using Lemma \ref{inqlmn} in \eqref{eq4.5} for $p>2$, we get
    \begin{align}\label{eq4.7}
        0 &\geq d_3 \bigg(\int_\Omega|\nabla (w_1-w_2)|^p \d x+{\int_{\mathbb{R}^N} \int_{\mathbb{R}^N}} \frac{|\bar{w}_1-\bar{w}_2|^p}{|x-y|^{N+sp}}\d x\d y +\beta\int_\Omega |w_1-w_2|^p \d x\bigg) \nonumber \\
        & = d_3 \|w_1-w_2\|_\beta^p.
    \end{align}
On the other hand, for $p\in(1,2]$, we use Lemma \ref{inqlmn} and reverse H\"older inequality in \eqref{eq4.5} to obtain
     \begin{align}\label{eq4.6}
         0  \geq & d_1 \bigg[ \int_\Omega |\nabla w_1-\nabla w_2|^2 (|\nabla w_1|+|\nabla w_2|)^{p-2} \d x\nonumber \\
         &+ {\int_{\mathbb{R}^N} \int_{\mathbb{R}^N}} \frac{|\bar{w}_1(x,y)- \bar{w}_2(x,y)|^2(|\bar{w}_1(x,y)|+|\bar{w}_2(x,y)|)^{p-2}}{|x-y|^{N+sp}} \d x \d y \nonumber \\
         & +\beta\int_\Omega |w_1-w_2|^2(|w_1|+|w_2|)^{p-2})\d x \bigg] \nonumber \\
         \geq & C_1 \bigg[ \left(\int_\Omega |\nabla w_1-\nabla w_2|^p \d x\right)^\frac{2}{p} \left(\int_\Omega (|\nabla w_1|+|\nabla w_2|)^p \d x\right)^\frac{p-2}{p} \nonumber\\
         &+ \bigg\{{\left( {\int_{\mathbb{R}^N} \int_{\mathbb{R}^N}}\frac{|\bar{w}_1(x,y)-\bar{w}_2(x,y)|^p}{|x-y|^{N+sp}}\d x \d y\right)^\frac{2}{p}}\nonumber \\ 
         & \hspace{0.5cm} \times {\left({\int_{\mathbb{R}^N} \int_{\mathbb{R}^N}} \frac{(|\bar{w}_1(x,y)|+|\bar{w}_2(x,y)|)^{p}}{|x-y|^{N+sp}}\d x \d y\right)^\frac{p-2}{p}} \bigg\} \nonumber \\
        &+ \beta \left( \int_\Omega |w_1-w_2|^p \d x \right)^\frac{2}{p} \left(\int_\Omega (|w_1|+|w_2|)^p \d x \right)^\frac{p-2}{p}\bigg]\nonumber \\
         = &C_1\bigg[\|\nabla(w_1-w_2)\|_p^2\left(\int_\Omega (|\nabla w_1|+|\nabla w_2|)^p \d x\right)^\frac{p-2}{p}\nonumber\\
        &+[w_1-w_2]_{s,p}^2\left({\int_{\mathbb{R}^N} \int_{\mathbb{R}^N}}\frac{(|\bar{w}_1(x,y)|+|\bar{w}_2(x,y)|)^{p}}{|x-y|^{N+sp}} \d x\right)^\frac{p-2}{p} \nonumber \\
        &+ \beta \|w_1-w_2\|_p^2 \left(\int_\Omega (|w_1|+|w_2|)^p \d x\right)^\frac{p-2}{p}\bigg] \geq C_2 \|w_1-w_2\|_\beta^2, 
     \end{align}
     
    where $C_1$ and $C_2$ are positive constants. Hence, from \eqref{eq4.7} and \eqref{eq4.6}, we get 
    $$\|w_1-w_2\|_\beta=0,$$ 
    implying that $w_1=w_2$. Therefore, we have a unique minimizer for $J$, that is, for a given $u\in \mathbb{X}_0^{s,p}(\Omega)$, the minimization problem $$m=\inf_{w\in \mathbb{X}_0^{s,p}(\Omega)}J(w)=J(w^*)$$ is well-defined with a unique minimizer. Thus, the operator $A$ is well-defined.\\
     \textbf{Claim II}: The operator $A$ is continuous.\\
        {\it Proof of the Claim II:} Let $(u_n)\subset \mathbb{X}_0^{s,p}(\Omega)$ be a sequence such that $u_n\rightarrow u\in \mathbb{X}_0^{s,p}(\Omega)$ strongly. Therefore, from Theorem \ref{thm cpt}, we have $u_n\rightarrow u$ strongly in $L^r(\Omega)$ for $r\in[1,p^*)$ and $u_n\rightarrow u$ $a.e.$ in $ \Omega$. Moreover, $(u_n)$ is uniformly bounded in $\mathbb{X}_0^{s,p}(\Omega)$. Let $A(u_n) = w_n$. On choosing $\phi = w_n$ in \eqref{WFNP}, we get
     \begin{align}\label{eq4.8}
         \|w_n\|^p_\beta=\int_\Omega g(x,u_n)w_n \d x.
     \end{align}
     From conditions $(f_1)$, $(f_2)$ and $(f_3)$, we conclude that $(w_n)$ is uniformly bounded in $ \mathbb{X}_0^{s,p}(\Omega)$, that is, there exists a constant $M>0$ such that $\|w_n\|_\beta\leq M$ for all $n\in\mathbb{N}$. Let us denote $\bar{w}_n(x,y)=w_n(x)-w_n(y)$ and $\bar{w}(x,y)=w(x)-w(y)$. On choosing $\phi=w_n-w$ as the test function in \eqref{WFNP} and using Lemma \ref{inqlmn} for $p>2$, we obtain
     \begin{align}\label{eq4.9}
         &\|w_n-w\|_\beta^p = \int_\Omega|\nabla (w_n-w)|^p \d x \nonumber \\
         &+{\int_{\mathbb{R}^N} \int_{\mathbb{R}^N}} \frac{|(w_n-w)(x)-(w_n-w)(y)|^p}{|x-y|^{N+sp}}\d x\d y +\beta\int_\Omega |w_n-w|^p \d x \nonumber \\
        \leq & C \bigg[ \int_\Omega (|\nabla w_n|^{p-2} \nabla w_n-|\nabla w|^{p-2} \nabla w)\cdot \nabla ({w}_n(x)-{w}(x))\d x \nonumber\\
      &+ {\int_{\mathbb{R}^N} \int_{\mathbb{R}^N}} \left( \frac{|\bar{w}_n(x,y)|^{p-2}\bar{w}_n(x,y)}{|x-y|^{N+sp}} - \frac{|\bar{w}(x,y)|^{p-2}\bar{w}(x,y)}{|x-y|^{N+sp}}\right)(\bar{w}_n(x,y)-\bar{w}(x,y))\d x \d y\nonumber \\ 
   & + \beta \int_\Omega (|w_n|^{p-2} w_n-|w|^{p-2} w)(w_n-w) \d x\bigg]\nonumber\\
   =&  C\int_\Omega (g(x,u_n)-g(x,u))(w_n-w) \d x= C I,
     \end{align}
     where $$I=\int_\Omega (g(x,u_n)-g(x,u))(w_n-w) \d x.$$
    For $p\in (1,2]$, again using Lemma \ref{inqlmn}, the reverse H\"older inequality and putting $\phi=w_n-w$ in \eqref{WFNP}, we obtain
     \begin{align}\label{eq4.10}
         I = & \int_\Omega (|\nabla w_n|^{p-2} \nabla w_n-|\nabla w|^{p-2} \nabla w)\cdot\nabla ({w}_n(x)-{w}(x))\d x \nonumber\\
     & + {\int_{\mathbb{R}^N} \int_{\mathbb{R}^N}}\left( \frac{|\bar{w}_n(x,y)|^{p-2}\bar{w}_n(x,y)}{|x-y|^{N+sp}} - \frac{|\bar{w}(x,y)|^{p-2}\bar{w}(x,y)}{|x-y|^{N+sp}}\right)(\bar{w}_n(x,y)-\bar{w}(x,y))\d x \d y\nonumber \\ 
   & + \beta \int_\Omega (|w_n|^{p-2} w_n-|w|^{p-2} w)(w_n-w) \d x \nonumber \\
   \geq & C \bigg[ \int_\Omega |\nabla w_n-\nabla w|^2 (|\nabla w_n|+|\nabla w|)^{p-2} \d x \nonumber \\
         &+ {\int_{\mathbb{R}^N} \int_{\mathbb{R}^N}} \frac{|\bar{w}_n(x,y)-\bar{w}(x,y)|^2(|\bar{w}_n(x,y)|+|\bar{w}(x,y)|)^{p-2}}{|x-y|^{N+sp}}\d x \d y \nonumber \\
         & +\beta\int_\Omega |w_n-w|^2(|w_n|+|w|)^{p-2})\d x \bigg] \nonumber \\
     \geq & C \bigg[ \left(\int_\Omega |\nabla w_n-\nabla w|^p\d x \right)^\frac{2}{p} \left(\int_\Omega (|\nabla w_n|+|\nabla w|)^p\d x  \right)^\frac{p-2}{p} \nonumber\\
        &+\bigg\{ \left( {\int_{\mathbb{R}^N} \int_{\mathbb{R}^N}}\frac{|\bar{w}_n(x,y)-\bar{w}(x,y)|^p}{|x-y|^{N+sp}}\d x \d y\right)^\frac{2}{p}\nonumber\\
        & \hspace{0.5cm} \times \left({\int_{\mathbb{R}^N} \int_{\mathbb{R}^N}} \frac{(|\bar{w}_n(x,y)|+|\bar{w}(x,y)|)^{p}}{|x-y|^{N+sp}}\d x \d y\right)^\frac{p-2}{p}\bigg\} \nonumber \\
        &+ \beta \left( \int_\Omega |w_n-w|^p\d x  \right)^\frac{2}{p} \left(\int_\Omega (|w_n|+|w|)^p \d x \right)^\frac{p-2}{p}\bigg]\nonumber \\
     = & C\bigg[\|\nabla(w_n-w)\|_p^2\left(\int_\Omega (|\nabla w_n|+|\nabla w|)^p \d x \right)^\frac{p-2}{p}\nonumber\\
        &+[w_n-w]_{s,p}^2\left({\int_{\mathbb{R}^N} \int_{\mathbb{R}^N}} \frac{(|\bar{w}_n(x,y)|+|\bar{w}(x,y)|)^{p}}{|x-y|^{N+sp}}\d x \d y\right)^\frac{p-2}{p} \nonumber \\
        &+ \beta \|w_n-w\|_p^2 \left(\int_\Omega (|w_n|+|w|)^p\d x \right)^\frac{p-2}{p}\bigg] \geq C \|w_n-w\|_\beta^2.
     \end{align}
     Now, let $R\geq r$ and define $h\in C_c^\infty(\mathbb{R})$ such that $0\leq h(t)\leq 1$, $\forall\, t \in \mathbb{R}$ and
     \begin{align}\nonumber
         h(t)=\begin{cases}
              & 1,~ \text{if}~ |t|\leq R,\\
                  & 0,~ \text{if}~ |t|\geq R+1. 
                  \end{cases}
       \end{align}
       We set $\mathcal{\phi}_{g,1}(t)=h(t)g(x,t)$ and $\mathcal{\phi}_{g,2}(t)=(1-h(t))g(x,t)$, $\forall \,t\in \mathbb{R}$. From the conditions $(f_1)$ and $(f_2)$, we conclude that there exists $C>0$ such that
       $$|\mathcal{\phi}_{g,1}(t)|\leq C |t|^{p-1}~\text{and}~|\mathcal{\phi}_{g,2}(t)|\leq C |t|^{q-1},~\forall \,t\in \mathbb{R}.$$ 
       Therefore, using the H\"older inequality and then using the Sobolev inequality, we deduce that
       \begin{align}\label{eq4.11}
           I&=\int_\Omega (g(x,u_n)-g(x,u))(w_n-w) \d x\nonumber \\
           &=\int_\Omega [(\mathcal{\phi}_{g,1}(u_n)-\mathcal{\phi}_{g,1}(u))+(\mathcal{\phi}_{g,2}(u_n)-\mathcal{\phi}_{g,2}(u))](w_n-w) \d x\nonumber \\
           &\leq \|\mathcal{\phi}_{g,1}(u_n)-\mathcal{\phi}_{g,1}(u)\|_\frac{p}{p-1}\|w_n-w\|_p + \|\mathcal{\phi}_{g,2}(u_n)-\mathcal{\phi}_{g,2}(u)\|_\frac{q}{q-1}\|w_n-w\|_q \nonumber\\
           & \leq C \Bigg[\|\mathcal{\phi}_{g,1}(u_n)-\mathcal{\phi}_{g,1}(u)\|_\frac{p}{p-1}+\|\mathcal{\phi}_{g,2}(u_n)-\mathcal{\phi}_{g,2}(u)\|_\frac{q}{q-1}\bigg]\|w_n-w\|_\beta. 
       \end{align}
       Hence, for $p>2$, we use \eqref{eq4.9} and \eqref{eq4.11} to obtain
       \begin{align}\label{eq4.12}
           \|w_n-w\|_\beta^{p-1}\leq C \bigg[ \|\mathcal{\phi}_{g,1}(u_n)-\mathcal{\phi}_{g,1}(u)\|_\frac{p}{p-1}+\|\mathcal{\phi}_{g,2}(u_n)-\mathcal{\phi}_{g,2}(u)\|_\frac{q}{q-1}\bigg ]
       \end{align}
       and if $1<p\leq 2$, then using \eqref{eq4.10} and \eqref{eq4.11}, we obtain,
       \begin{align}\label{eq4.12A}
           \|w_n-w\|_\beta^{}\leq C \bigg[ \|\mathcal{\phi}_{g,1}(u_n)-\mathcal{\phi}_{g,1}(u)\|_\frac{p}{p-1}+\|\mathcal{\phi}_{g,2}(u_n)-\mathcal{\phi}_{g,2}(u)\|_\frac{q}{q-1}\bigg ].
       \end{align}
       We now apply the Lebesgue-dominated convergence theorem in the inequalities \eqref{eq4.12} and \eqref{eq4.12A}  to obtain $\|w_n-w\|_\beta\rightarrow 0$ as $n\rightarrow\infty,$ concluding the continuity of the operator $A$.\\
       \textbf{Claim III}: The operator $A$ is compact.\\
        {\it Proof of the Claim III:} Let $(u_n)\subset \mathbb{X}_0^{s,p}(\Omega)$ be uniformly bounded and $w_n=A(u_n)$. Following the arguments as in the proof of {\textit{Claim II}}, we assert that $(w_n)$ is uniformly bounded in $\mathbb{X}_0^{s,p}(\Omega)$, that is, there exists $M>0$ such that $\|w_n\|_\beta\leq M$ for all $n\in\mathbb{N}$. Therefore, there exist subsequences of $(u_n)$ and $(w_n)$, still denoted by $(u_n)$ and $(w_n)$, respectively such that $u_n \rightharpoonup \bar{u}\in \mathbb{X}_0^{s,p}(\Omega)$ and $w_n \rightharpoonup \bar{w}\in \mathbb{X}_0^{s,p}(\Omega)$ weakly in $\mathbb{X}_0^{s,p}(\Omega).$ Thus from Theorem \ref{thm cpt}, we get $u_n\rightarrow\bar{u}$; $w_n\rightarrow\bar{w}$ strongly in $L^r(\Omega)$ for $r\in [1,p^*)$, $u_n\rightarrow \bar{u}$; $w_n \rightarrow \bar{w}$ $a.e.$ in $\Omega.$ Again, repeating the arguments of the proof of {\textit{Claim II}}, we deduce that $\|w_n - \bar{w}\|_\beta \rightarrow 0 $ as $ n \rightarrow \infty $. This completes the proof.
    \end{proof}
\begin{lemma}\label{lmn4.2}
     There exist two positive constants $a_1$ and $a_2$ such that
    \begin{align}\label{eq4.13L}
    \langle \Phi'(u),u-A(u) \rangle \geq\begin{cases}
         & a_1 \|u-A(u)\|^p_{\mathbb{X}_0^{s,p}(\Omega)},~\text{if}~p>2 \\
         & a_2\dfrac{\|u-A(u)\|^2_{\mathbb{X}_0^{s,p}(\Omega)}}{\left(\|u\|_{\mathbb{X}_0^{s,p}(\Omega)}+\|A(u)\|_{\mathbb{X}_0^{s,p}(\Omega)}\right)^{2-p}}, ~\text{if}~1<p\leq2.
    \end{cases}
    \end{align}
\end{lemma}
\begin{proof}
    Let $u\in \mathbb{X}_0^{s,p}(\Omega)$. We set $A(u)=w$, $\bar{u}(x,y)=u(x)-u(y)$, and $\bar{w}(x,y)=w(x)-w(y)$. Suppose $p>2$. Then  using Lemma \ref{inqlmn} and putting $\phi=u-w$ in \eqref{WFNP}, we obtain
    {\begin{align}\label{eq4.13}
        & \langle \Phi'(u),u-A(u) \rangle \\
        = & \int_\Omega |\nabla u|^{p-2} \nabla u\cdot\nabla(u-w)\d x \nonumber\\
        &+ {\int_{\mathbb{R}^N} \int_{\mathbb{R}^N}} \frac{|\bar{u}(x,y)|^{p-2}\bar{u}(x,y)(\bar{u}(x,y)-\bar{w}(x,y))}{|x-y|^{N+sp}}\d x \d y -\int_\Omega f(x,u)(u-w) \d x \nonumber \\
        = & \int_\Omega \left[|\nabla u|^{p-2} \nabla u-|\nabla w|^{p-2} \nabla w \right]\cdot \nabla(u-w)\d x \nonumber\\
        &+ {\int_{\mathbb{R}^N} \int_{\mathbb{R}^N}} \frac{\left[|\bar{u}(x,y)|^{p-2}\bar{u}(x,y)-|\bar{w}(x,y)|^{p-2}\bar{w}(x,y)\right](\bar{u}(x,y)-\bar{w}(x,y))}{|x-y|^{N+sp}} \d x \d y\nonumber \\
        &+ \beta \int_\Omega (|u|^{p-2}u-|w|^{p-2} w)(u-w) \d x\nonumber \\
        \geq & C \Bigg[\int_\Omega|\nabla (u-w)|^p \d x +{\int_{\mathbb{R}^N} \int_{\mathbb{R}^N}} \frac{|(u-w)(x)-(u-w)(y)|^p}{|x-y|^{N+sp}} \d x \d y+\beta\int_\Omega |u-w|^p \d x\bigg] \nonumber \\
        \geq & a_1 \|u-w\|^p_{\mathbb{X}_0^{s,p}(\Omega)},
    \end{align}}
    
    \noindent for some constant $a_1>0$. This proves \eqref{eq4.13L} for $p>2$. For $1<p\leq 2$, we apply Lemma \ref{inqlmn} again with $\phi=u-w$ as the test function in \eqref{WFNP} to obtain
    \begin{align}\label{eq4.14}
        & \langle \Phi'(u),u-A(u) \rangle \nonumber \\
        = & \int_\Omega \left[|\nabla u|^{p-2} \nabla u-|\nabla w|^{p-2} \nabla w \right] \cdot \nabla(u-w) \d x\nonumber\\
        & + {\int_{\mathbb{R}^N} \int_{\mathbb{R}^N}} \frac{\left[|\bar{u}(x,y)|^{p-2}\bar{u}(x,y)-|\bar{w}(x,y)|^{p-2}\bar{w}(x,y)\right](\bar{u}(x,y)-\bar{w}(x,y))}{|x-y|^{N+sp}}\d x\d y \nonumber \\
        &+ \beta \int_\Omega (|u|^{p-2}u-|w|^{p-2} w)(u-w)\d x \nonumber \\
        \geq & C \bigg[\int_\Omega |\nabla u-\nabla w|^2 (|\nabla u|+|\nabla w|)^{p-2} \d x \nonumber \\
         &+ {\int_{\mathbb{R}^N} \int_{\mathbb{R}^N}} \frac{|\bar{u}(x,y)-\bar{w}(x,y)|^2(|\bar{u}(x,y)|+|\bar{w}(x,y)|)^{p-2}}{|x-y|^{N+sp}}\d x \d y \nonumber \\
         & +\beta\int_\Omega |u-w|^2(|u|+|w|)^{p-2})\d x\bigg].
    \end{align}
We now employ the H\"older inequality and the Minkowski inequality to obtain the following estimate;
\begin{align}\label{eq4.15}
   \int_\Omega &|\nabla (u-w)|^p \d x+ {\int_{\mathbb{R}^N} \int_{\mathbb{R}^N}} \frac{|\bar{u}(x,y)-\bar{w}(x,y)|^p}{|x-y|^{N+sp}}\d x \d y \nonumber \\
   \leq & C \int_\Omega |\nabla (u-w)|^p \d x = C \int_\Omega |\nabla (u-w)|^p(|\nabla u|+|\nabla w|)^\frac{p(p-2)}{2} (|\nabla u|+|\nabla w|)^\frac{p(2-p)}{2}\d x \nonumber \\
   \leq & C \bigg( \int_\Omega |\nabla (u-w)|^2(|\nabla u|+|\nabla w|)^{p-2}\d x\bigg)^\frac{p}{2} \bigg( \int_\Omega (|\nabla u|+|\nabla w|)^p \d x\bigg)^\frac{2-p}{2} \nonumber \\
   \leq & C\bigg( \int_\Omega |\nabla (u-w)|^2(|\nabla u|+|\nabla w|)^{p-2}\d x\bigg)^\frac{p}{2} \bigg(\|\nabla u\|_p+\|\nabla w\|_p\bigg)^\frac{p(2-p)}{2} \nonumber \\
   \leq & C \bigg( \int_\Omega |\nabla (u-w)|^2(|\nabla u|+|\nabla w|)^{p-2}\d x\bigg)^\frac{p}{2} \bigg(\| u\|_{\mathbb{X}_0^{s,p}(\Omega)}+\| w\|_{\mathbb{X}_0^{s,p}(\Omega)}\bigg)^\frac{p(2-p)}{2},
\end{align}
for some C. Therefore, inequality \eqref{eq4.15} implies that 
\begin{align}\label{eq4.16}
    \|u-w\|_{\mathbb{X}_0^{s,p}(\Omega)}^2\leq C^\frac{2}{p} \bigg\{\bigg( \int_\Omega |\nabla (u-w)|^2(|\nabla u|+|\nabla w|)^{p-2}\d x \bigg)\nonumber\\
    \times \bigg(\| u\|_{\mathbb{X}_0^{s,p}(\Omega)}+\| w\|_{\mathbb{X}_0^{s,p}(\Omega)}\bigg)^{2-p}\bigg\}.
\end{align}

 \noindent Thus, from \eqref{eq4.14} and \eqref{eq4.16}, we conclude that there exists a positive constant $a_2$ such that \eqref{eq4.13L} holds true for $1<p\leq 2$. This completes the proof.
\end{proof}

\begin{lemma}\label{lmn4.3}
    There exist two constants $a_3, a_4>0$  such that
    {\small\begin{align}\label{eq4.19}
        \|\Phi'(u)\|_{{\mathbb{X}_0^{s,p}(\Omega)}^*} \leq \begin{cases}
             a_3\|u-A(u)\|_{\mathbb{X}_0^{s,p}(\Omega)}\left(\|u\|_{\mathbb{X}_0^{s,p}(\Omega)}+\|A(u)\|_{\mathbb{X}_0^{s,p}(\Omega)}\right)^{p-2}, &\text{if}~p>2\\
             a_4\|u-A(u)\|_{\mathbb{X}_0^{s,p}(\Omega)}^{p-1}, &\text{if}~1<p\leq2.
        \end{cases}
    \end{align}}
\end{lemma}
\begin{proof}
    Let $u, \phi\in \mathbb{X}_0^{s,p}(\Omega)$. We denote $A(u)=w$, $\bar{u}(x,y)=u(x)-u(y)$, $\bar{w}(x,y)=w(x)-w(y)$ and $\bar{\phi}(x,y) = \phi(x)-\phi(y)$. By applying the H\"older's inequality, we obtain
    \begin{align}\label{eq4.20}
        & \langle \Phi'(u),\phi \rangle \nonumber \\
        = & \int_\Omega |\nabla u|^{p-2} \nabla u \cdot\nabla\phi\d x+ {\int_{\mathbb{R}^N} \int_{\mathbb{R}^N}} \frac{|\bar{u}(x,y)|^{p-2}\bar{u}(x,y)\bar{\phi}(x,y)}{|x-y|^{N+sp}}\d x \d y -\int_\Omega f(x,u)\phi \d x \nonumber \\ 
        = & \int_\Omega \left[|\nabla u|^{p-2} \nabla u-|\nabla w|^{p-2} \nabla w \right]\cdot \nabla\phi \d x \nonumber\\
        &+ {\int_{\mathbb{R}^N} \int_{\mathbb{R}^N}} \frac{\left[|\bar{u}(x,y)|^{p-2}\bar{u}(x,y)-|\bar{w}(x,y)|^{p-2}\bar{w}(x,y)\right]\bar{\phi}(x,y)}{|x-y|^{N+sp}} \d x  \d y\nonumber \\
        &+ \beta \int_\Omega (|u|^{p-2}u-|w|^{p-2} w)\phi \d x \nonumber \\
        \leq & \bigg( \int_\Omega| |\nabla u|^{p-2} \nabla u-|\nabla w|^{p-2} \nabla w|^\frac{p}{p-1} \d x\bigg)^\frac{p-1}{p}\bigg( \int_\Omega |\nabla \phi|^p\d x \bigg)^\frac{1}{p}\nonumber \\
        &+ \bigg\{ \bigg( {\int_{\mathbb{R}^N} \int_{\mathbb{R}^N}} \frac{\|\bar{u}(x,y)|^{p-2}\bar{u}(x,y)-|\bar{w}(x,y)|^{p-2}\bar{w}(x,y)|^\frac{p}{p-1}}{|x-y|^{N+sp}}\d x\d y\bigg)^\frac{p-1}{p}
        \nonumber \\ 
         & \hspace{0.5cm} \times \bigg({\int_{\mathbb{R}^N} \int_{\mathbb{R}^N}} \frac{|\bar{\phi}(x,y)|^p}{|x-y|^{N+sp}}\d x\d y\bigg)^\frac{1}{p} \bigg\}\nonumber\\
         &+ \beta \bigg(\int_\Omega \|u|^{p-2}u-|w|^{p-2} w|^\frac{p}{p-1}\d x\bigg)^\frac{p-1}{p}\bigg(\int_\Omega|\phi|^p\d x\bigg)^\frac{1}{p}\nonumber \\
         \leq & C \Bigg[ \bigg( \int_\Omega| |\nabla u|^{p-2} \nabla u-|\nabla w|^{p-2} \nabla w|^\frac{p}{p-1} \d x\bigg)^\frac{p-1}{p}+\beta \bigg(\int_\Omega \|u|^{p-2}u-|w|^{p-2} w|^\frac{p}{p-1}\d x\bigg)^\frac{p-1}{p} \nonumber \\
         &+\bigg({\int_{\mathbb{R}^N} \int_{\mathbb{R}^N}} \frac{\|\bar{u}(x,y)|^{p-2}\bar{u}(x,y)-|\bar{w}(x,y)|^{p-2}\bar{w}(x,y)|^\frac{p}{p-1}}{|x-y|^{N+sp}}\d x\d y\bigg)^\frac{p-1}{p} \bigg]\|\phi\|_{\mathbb{X}_0^{s,p}(\Omega)}. 
         \end{align}
        Therefore, from \eqref{eq4.20}, we get 
         \begin{align}\label{eq4.21}
            \|\Phi'(u)\|_{{\mathbb{X}_0^{s,p}(\Omega)}^*}
            \leq &C \Bigg[ \bigg( \int_\Omega| |\nabla u|^{p-2} \nabla u-|\nabla w|^{p-2} \nabla w|^\frac{p}{p-1}\d x \bigg)^\frac{p-1}{p}\nonumber \\
         &+\bigg({\int_{\mathbb{R}^N} \int_{\mathbb{R}^N}} \frac{\|\bar{u}(x,y)|^{p-2}\bar{u}(x,y)-|\bar{w}(x,y)|^{p-2}\bar{w}(x,y)|^\frac{p}{p-1}}{|x-y|^{N+sp}}\d x \d y\bigg)^\frac{p-1}{p} \nonumber\\
         & +\beta \bigg(\int_\Omega \|u|^{p-2}u-|w|^{p-2} w|^\frac{p}{p-1}\bigg)^\frac{p-1}{p} \d x\bigg]. 
         \end{align}
        Now for $p>2$, recall Lemma \ref{inqlmn}. Thus, using the H\"older and Minkowski inequalities in \eqref{eq4.21}, we derive
\begin{align}
\|\Phi'(u)\|_{{\mathbb{X}_0^{s,p}(\Omega)}^*}\leq & C \bigg[ \bigg(\int_\Omega(|\nabla u|+|\nabla w|)^\frac{p(p-2)}{p-1}|\nabla u-\nabla w|^\frac{p}{p-1}\d x\bigg)^\frac{p-1}{p}\nonumber\\
& +\bigg({\int_{\mathbb{R}^N} \int_{\mathbb{R}^N}} \frac{(|\bar{u}(x,y)|+|\bar{w}(x,y|)^\frac{p(p-2)}{p-1}|\bar{u}(x,y)-\bar{w}(x,y)|^\frac{p}{p-1}}{|x-y|^{N+sp}}\d x \d y\bigg)^\frac{p-1}{p}\nonumber\\
& + \beta \bigg( \int_\Omega(|u|+|w|)^\frac{p(p-2)}{p-1}|u-w|^\frac{p}{p-1}\d x \bigg)^\frac{p-1}{p}\bigg]\nonumber\\
\leq & C \bigg[ \bigg(\int_\Omega(|\nabla u|+|\nabla w|)^p \d x \bigg)^\frac{p-2}{p} \bigg( \int_\Omega |\nabla u-\nabla w|^p \d x\bigg)^\frac{1}{p} \nonumber \\
 & +\bigg\{ \bigg({\int_{\mathbb{R}^N} \int_{\mathbb{R}^N}} \frac{(|\bar{u}(x,y)|+|\bar{w}(x,y|)^p}{|x-y|^{N+sp}}\d x \d y\bigg)^\frac{p-2}{p}\nonumber \\
 &\hspace{0.5cm} \times \bigg({\int_{\mathbb{R}^N} \int_{\mathbb{R}^N}} \frac{|\bar{u}(x,y)-\bar{w}(x,y)|^p}{|x-y|^{N+sp}}\d x \d y\bigg)^\frac{1}{p} \bigg\} \nonumber\\
& + \beta \bigg( \int_\Omega(|u|+|w|)^p\d x\bigg)^\frac{p-2}{p}\bigg(\int_\Omega|u-w|^p \d x \bigg)^\frac{1}{p}\bigg]\nonumber\\
\leq & a_3\left(\|u\|_{\mathbb{X}_0^{s,p}(\Omega)}+\|A(u)\|_{\mathbb{X}_0^{s,p}(\Omega)}\right)\|u-A(u)\|_{\mathbb{X}_0^{s,p}(\Omega)},
\end{align}
for some constant $a_3>0$, proving \eqref{eq4.19} for $p>2$. When $1<p\leq 2$, again apply Lemma \ref{inqlmn} in \eqref{eq4.21} to obtain
\begin{align}
    \|\Phi'(u)\|_{{\mathbb{X}_0^{s,p}(\Omega)}^*}\leq & C \bigg[  \bigg( \int_\Omega |\nabla u-\nabla w|^p \d x\bigg)^\frac{p-1}{p} + \bigg({\int_{\mathbb{R}^N} \int_{\mathbb{R}^N}} \frac{|\bar{u}(x,y)-\bar{w}(x,y)|^p}{|x-y|^{N+sp}}\d x \d y\bigg)^\frac{p-1}{p}\nonumber\\
& + \beta \bigg(\int_\Omega|u-w|^p \d x \bigg)^\frac{p-1}{p}\bigg]\nonumber\\
\leq & a_4 \|u-A(u)\|_{\mathbb{X}_0^{s,p}(\Omega)}^{p-1},
\end{align}
for some $a_4>0.$ This completes the proof.
    \end{proof}
\begin{lemma}\label{lmn4.4}
    Let $p>1$. For all $u\in \mathbb{X}_0^{s,p}(\Omega)$ with $\Phi(u)\leq c$ for any $c\in \mathbb{R}$, there exists a constant $a_5>0$ (depending only on $c$), such that 
    \begin{equation}\label{nn}
        \|u\|_{\mathbb{X}_0^{s,p}(\Omega)}+\|A(u)\|_{\mathbb{X}_0^{s,p}(\Omega)}\leq a_5\left(1+\|u-A(u)\|_{\mathbb{X}_0^{s,p}(\Omega)}\right).
    \end{equation} 
\end{lemma}
\begin{proof}
Let $u\in \mathbb{X}_0^{s,p}(\Omega)$ and $A(u)=w$. Therefore, we have
    \begin{align}\label{eq4.24}
        \Phi(u)-\frac{1}{\mu} \langle \Phi'(u),u \rangle = & \left( \frac{1}{p}-\frac{1}{\mu}\right) \bigg( \int_\Omega |\nabla u|^{p} \d x+{\int_{\mathbb{R}^N} \int_{\mathbb{R}^N}} \frac{|u(x)-u(y)|^{p}}{|x-y|^{N+sp}}\d x \d y\bigg)\nonumber \\
        & - \int_\Omega\bigg(F(x,u)-\frac{1}{\mu}f(x,u)u \bigg)\d x.
    \end{align}
    For $\Phi(u)\leq c$, the condition $(f_3)$ and \eqref{eq4.24} imply that 
    \begin{align}\label{eq4.25}
        C \|u\|_{\mathbb{X}_0^{s,p}(\Omega)}^p \leq \Phi(u)-\frac{1}{\mu} \langle \Phi'(u),u \rangle, 
    \end{align}
    Again, from $\eqref{eq4.25}$, we get
    \begin{align}\label{eq4.26}
        \|u\|_{\mathbb{X}_0^{s,p}(\Omega)}^p \leq C\left(1+\|\Phi'(u)\|_{{\mathbb{X}_0^{s,p}(\Omega)}^*}\|u\|_{\mathbb{X}_0^{s,p}(\Omega)}\right).
    \end{align}
    If $1<p\leq2$, we use the inequalities \eqref{eq4.19} and \eqref{eq4.26} to obtain
    \begin{align}\label{eq4.27}
         \|u\|_{\mathbb{X}_0^{s,p}(\Omega)}^p \leq C\left(1+\|u-A(u)\|_{\mathbb{X}_0^{s,p}(\Omega)}^{p-1}\|u\|_{\mathbb{X}_0^{s,p}(\Omega)}\right).
    \end{align}
    Applying the Young's inequality in \eqref{eq4.27}, we have
    \begin{align}\label{eq4.28}
        \|u\|_{\mathbb{X}_0^{s,p}(\Omega)} \leq C\left(1+\|u-A(u)\|_{\mathbb{X}_0^{s,p}(\Omega)}\right).
    \end{align}
    Therefore, the continuity of the operator $A$ and the inequality \eqref{eq4.28} leads to the result. Similarly, for $p>2$ we use \eqref{eq4.19} and \eqref{eq4.26} to obtain
    \begin{align}\label{eq4.29}
        \|u\|_{\mathbb{X}_0^{s,p}(\Omega)}^p \leq C \left(1+\|u-A(u)\|_{\mathbb{X}_0^{s,p}(\Omega)}\left(\|u\|_{\mathbb{X}_0^{s,p}(\Omega)}+\|A(u)\|_{\mathbb{X}_0^{s,p}(\Omega)}\right)^{p-2}\|u\|_{\mathbb{X}_0^{s,p}(\Omega)}\right).
    \end{align}
    Thus using the Young's inequality in \eqref{eq4.29} with $p'=\frac{p}{p-1}$, we get
    \begin{align}\label{eq4.30}
        \|u\|_{\mathbb{X}_0^{s,p}(\Omega)}^p \leq C \left(1+\|u-A(u)\|^{p'}_{\mathbb{X}_0^{s,p}(\Omega)}\left(\|u\|_{\mathbb{X}_0^{s,p}(\Omega)}+\|A(u)\|_{\mathbb{X}_0^{s,p}(\Omega)}\right)^{p-p'}\right).
    \end{align}
    This implies that
    \begin{align}\label{eq4.31}
         \|u\|_{\mathbb{X}_0^{s,p}(\Omega)} \leq C \left(1+\|u-A(u)\|^\frac{p'}{p}_{\mathbb{X}_0^{s,p}(\Omega)}\left(\|u\|_{\mathbb{X}_0^{s,p}(\Omega)}+\|A(u)\|_{\mathbb{X}_0^{s,p}(\Omega)}\right)^{1-\frac{p'}{p}}\right).
    \end{align}
    Since $A$ is continuous, the inequality \eqref{eq4.31} and Young's inequality conclude the estimate \eqref{nn} for $p>2$, completing the proof.
\end{proof}
\begin{lemma}\label{lmn4.5}
    Let $a<b$ and $a_1>0$. For every $u\in \Phi^{-1}[a,b]$ with $\|\Phi'(u)\|_{{\mathbb{X}_0^{s,p}(\Omega)}^*}\geq a_1$, there exists $b_1>0$ such that 
    $$\|u-A(u)\|_{\mathbb{X}_0^{s,p}(\Omega)}\geq b_1,$$
\end{lemma}
\begin{proof}
    Suppose the conclusion fails. Then we assume by contradiction that there exists a sequence $(u_n)\subset \mathbb{X}_0^{s,p}(\Omega)$ that satisfying $(u_n)\subset \Phi^{-1}[a,b]$ and $\|\Phi'(u_n)\|_{{\mathbb{X}_0^{s,p}(\Omega)}^*}\geq a_1$ such that $$\|u_n-A(u_n)\|\rightarrow 0.$$
    For $p>2$, Lemma \ref{lmn4.3} and Lemma \ref{lmn4.4} conclude that
    \begin{align}\label{eq4.32}
        \|\Phi'(u_n)\|_{{\mathbb{X}_0^{s,p}(\Omega)}^*} & \leq a_3\|u_n-A(u_n)\|_{\mathbb{X}_0^{s,p}(\Omega)}\left(\|u_n\|_{\mathbb{X}_0^{s,p}(\Omega)}+\|A(u_n)\|_{\mathbb{X}_0^{s,p}(\Omega)}\right)^{p-2}\nonumber\\
        & \leq a_3a_5^{p-2}\|u_n-A(u_n)\|_{\mathbb{X}_0^{s,p}(\Omega)}\left(1+\|u_n-A(u_n)\|_{\mathbb{X}_0^{s,p}(\Omega)}\right)^{p-2},
    \end{align}
    implying that $\|\Phi'(u_n)\|_{{\mathbb{X}_0^{s,p}(\Omega)}^*}\rightarrow 0$. Thus we arrive at a contradiction. Again, for $1<p\leq2$, from Lemma \ref{lmn4.3}, we get
    \begin{align}\label{4.33}
        \|\Phi'(u_n)\|_{{\mathbb{X}_0^{s,p}(\Omega)}^*} \leq a_4\|u_n-A(u_n)\|_{\mathbb{X}_0^{s,p}(\Omega)}^{p-1}.
    \end{align}
    Observe that \eqref{4.33} gives $\|\Phi'(u_n)\|_{{\mathbb{X}_0^{s,p}(\Omega)}^*}\rightarrow 0$, which is a contradiction. This completes the proof.
\end{proof}
\begin{lemma}\label{lmn4.6}
    There exists $\epsilon_0$, (small enough), such that for every $\epsilon \in(0,\epsilon_0)$, we have
    \begin{align}\label{eq4.34}
        \dist_\beta(A(u),P^-)& <\theta \epsilon,~ \forall\, u\in \bar{P}^-_\epsilon~\text{and}~
        \dist_\beta(A(u),P^+) <\theta \epsilon,~ \forall\,u\in \bar{P}^+_\epsilon,
    \end{align}
    for some $\theta\in(0,1).$
\end{lemma}
\begin{proof}
   We only prove the estimate $\dist_\beta(A(u),P^-) <\theta \epsilon,~ \forall\, u\in \bar{P}^-_\epsilon$  and the other one follows similarly. Let $u\in \bar{P}_\epsilon^-$
   and define $A(u)=w$. From $(f_1)$ and $(f_2)$, we deduce that for every $\delta>0$, there exists $C_\delta>0$ such that
   \begin{align}\label{eq4.36}
   |f(x,t)|\leq \delta|t|^{p-1}+C_\delta|t|^{q-1},~\forall\,x\in \bar{\Omega}~\text{and  } \forall\,t \in \mathbb{R}.
    \end{align}
   On using the inequality \eqref{eq4.36} in \eqref{WFNP} taking $w^+$ as the test function, it yields
   \begin{align}\label{eq4.37}
       &\int_\Omega |\nabla w|^{p-2} \nabla w \cdot \nabla w^+ \d x  +\beta \int_\Omega |w|^{p-2} w w^+ \d x \nonumber \\
       &+ {\int_{\mathbb{R}^N} \int_{\mathbb{R}^N}} \frac{|w(x)-w(y)|^{p-2}(w(x)-w(y))(w^+(x)-w^+(y))}{|x-y|^{N+sp}} \d x \d y  \nonumber\\
   =& \int_\Omega g(x,u)w^+ \d x\nonumber \\
   \leq &\int_\Omega (f(x,u^+)+\beta|u^+|^{p-2}u^+)w^+ \d x \nonumber \\
     \leq& \int_\Omega ((\delta|u^+|^{p-1}+C_\delta|u^+|^{q-1})+\beta|u^+|^{p-1})w^+ \d x = \int_\Omega ((\delta+\beta)|u^+|^{p-1}+C_\delta|u^+|^{q-1})w^+ \d x \nonumber \\
    \leq &(\delta+\beta)\|u^+\|^{p-1}_p\|w^+\|_p +C_\delta\|u^+\|^{q-1}_q\|w^+\|_q \nonumber\\
    \leq &\frac{\delta+\beta}{(\lambda_1+\lambda_{1,s}+\beta)^\frac{1}{p}}\|u^+\|^{p-1}_p\|w^+\|_\beta +CC_\delta \|u^+\|^{q-1}_q\|w^+\|_\beta,
   \end{align}
   where $\lambda_1$ and $\lambda_{1,s}$ are the first eigenvalues of $-\Delta_p$ and $(-\Delta_p)^s$, respectively. Now the left-hand side of \eqref{eq4.37} implies that 
   \begin{align}\label{eq4.38}
        & \int_\Omega|\nabla w|^{p-2} \nabla w \cdot \nabla w^+ \d x +\beta \int_\Omega |w|^{p-2} w w^+\d x \nonumber \\
       &+ {\int_{\mathbb{R}^N} \int_{\mathbb{R}^N}} \frac{|w(x)-w(y)|^{p-2}(w(x)-w(y))(w^+(x)-w^+(y))}{|x-y|^{N+sp}} \d x \d y \nonumber\\
        \geq  &\int_\Omega |\nabla w^+|^{p} \d x+ {\int_{\mathbb{R}^N} \int_{\mathbb{R}^N}} \frac{|w^+(x)-w^+(y)|^{p}}{|x-y|^{N+sp}} \d x \d y +\beta \int_\Omega |w^+|^{p} \d x =\|w^+\|_\beta^p.
   \end{align}
   On using \eqref{eq4.37} and \eqref{eq4.38} together with the Poincar\'e inequality and mixed Sobolev inequality, we obtain
   \begin{align}\label{eq4.39}
       \dist_\beta(w,P^-)^{p-1}\leq \|w^+\|_\beta^{p-1}& \leq \frac{\delta+\beta}{(\lambda_1+\lambda_{1,s}+\beta)^\frac{1}{p}}\|u^+\|^{p-1}_p +CC_\delta \|u^+\|^{q-1}_q \nonumber\\
       & \leq \frac{\delta+\beta}{(\lambda_1+\lambda_{1,s}+\beta)}\|u^+\|^{p-1}_\beta +CC_\delta \|u^+\|^{q-1}_\beta.
   \end{align}
   Note that for any $v\in P^-$, we have $\|u^+\|_r\leq \|u-v\|_r$, $\forall\,r\geq 1$. Therefore, from \eqref{eq4.39}, we derive that
   \begin{align}
        \dist_\beta(w,P^-)^{p-1}\leq \frac{\delta+\beta}{(\lambda_1+\lambda_{1,s}+\beta)}\|u-v\|^{p-1}_\beta +CC_\delta \|u-v\|^{q-1}_\beta,~\forall\,v\in P^-,
   \end{align}
   which implies
   \begin{align}
       \dist_\beta(w,P^-)^{p-1}\leq \frac{\delta+\beta}{(\lambda_1+\lambda_{1,s}+\beta)}\dist_\beta(u,P^-)^{p-1} +CC_\delta \dist_\beta(u,P^-)^{q-1}.
   \end{align}
  We now choose $\delta\in(0,\frac{\lambda_1+\lambda_{1,s}}{1+2(\lambda_1+\lambda_{1,s}+\beta)})$, so that
  \begin{align}\label{eq4.43}
      \dist_\beta(w,P^-)^{p-1}\leq (1-2\delta)\dist_\beta(u,P^-)^{p-1} +CC_\delta \dist_\beta(u,P^-)^{q-1}.
  \end{align}
  Let $0<\epsilon_0 <\left( \frac{\delta}{C C_\delta} \right)^{\frac{1}{q-p}}$ and $\dist_\beta(u,P^-)<\epsilon \leq \epsilon_0$. Therefore, from \eqref{eq4.43}, we get
  %$\epsilon_0 \in (0,\left \frac{\delta}{CC_\delta} \right)^{\frac{1}{q-p}} $  From \eqref{eq4.43}, we get
  \begin{align*}
      \dist_\beta(w,P^-)^{p-1}\leq (1-\delta) \dist_\beta(u,P^-)^{p-1}.
  \end{align*}
  Thus, we have
  \begin{align}\label{4.43}
  \dist_\beta(w,P^-)\leq (1-\delta)^{\frac{1}{p-1}} \dist_\beta(u,P^-)<\theta\epsilon<\epsilon,
  \end{align}
  where $\theta=(1-\delta)^{\frac{1}{p-1}}$. This concludes that $w=A(u)\in P^-_{\theta \epsilon}$. Therefore, if $u\in {P}^-_\epsilon$ such that $\epsilon \in (0, \epsilon_0)$ is a nontrivial solution to the problem \eqref{MP}, then, we have $u=A(u)$. Hence, from \eqref{4.43}, we deduce $u\in P^-$. This completes the proof.
\end{proof}
In the next theorem, we construct a Lipschitz continuous map, $B$, which serves as a pseudo-gradient vector field for $\Phi$.
\begin{lemma}\label{lmn4.7}
    There exists a locally Lipschitz continuous operator $B:E=:{\mathbb{X}_0^{s,p}(\Omega)}\setminus M\rightarrow \mathbb{X}_0^{s,p}(\Omega)$ with the following properties:\\
    $(a)$ There exists $\epsilon_1$ (small enough) such that, for every $\epsilon\in(0,\epsilon_1)$ and for some $\theta_1\in(0,1)$,
    \begin{align}\label{eq4.40}
        \dist_\beta(B(u),P^-)& <\theta_1 \epsilon,~ \forall\, \bar{P}^-_\epsilon~\text{and}\\ \label{eq4.41}
        \dist_\beta(B(u),P^+)& <\theta_1 \epsilon,~ \forall\, \bar{P}^+_\epsilon,
    \end{align}
    $(b)$ For all $u\in E,$
    \begin{align}\label{eq4.42}
    \frac{1}{2}\|u-B(u)\|_{\mathbb{X}_0^{s,p}(\Omega)}\leq\|u-A(u)\|_{\mathbb{X}_0^{s,p}(\Omega)}\leq\|u-B(u)\|_{\mathbb{X}_0^{s,p}(\Omega)},
    \end{align}
     $(c)$ For all $u\in E,$
    \begin{align}\label{eq4.33}
    \langle \Phi'(u),u-B(u) \rangle \geq    
    \begin{cases}
          & \frac{a_1}{2} \|u-A(u)\|^p_{\mathbb{X}_0^{s,p}(\Omega)},~\text{if}~p>2\\
         & \frac{a_2}{2}\frac{\|u-A(u)\|^2_{\mathbb{X}_0^{s,p}(\Omega)}}{\left(\|u\|_{\mathbb{X}_0^{s,p}(\Omega)}+\|A(u)\|_{\mathbb{X}_0^{s,p}(\Omega)}\right)^{2-p}}, ~\text{if}~p\in(1,2],
     \end{cases}
      \end{align}
      $(d)$ If $f$ is odd, then $B$ is odd.
\end{lemma}
\begin{proof}
We follow \cite[Lemma 4.1]{BL2004}. We only sketch the proof for the range $1<p\leq 2$. Let us define $\chi_1,\chi_2 \in C(E,\mathbb{R})$ as follows:
\begin{align}
    \chi_1(u)&=\frac{1}{2}||u-A(u)||_{\mathbb{X}_0^{s,p}(\Omega)}\label{4.48}~\text{and}\\
    \chi_2(u)&=\frac{a_2}{2a_4}||u-A(u)||_{\mathbb{X}_0^{s,p}(\Omega)}^{3-p}(||u||_{\mathbb{X}_0^{s,p}(\Omega)}+||A(u)||_{\mathbb{X}_0^{s,p}(\Omega)}),\label{4.49}
\end{align}
where $a_2$ and $a_4$ are identical as in Lemma \ref{lmn4.2} and Lemma \ref{lmn4.3} respectively.
For each $u\in E$, there exists $\gamma(u)>0$ such that for every $v_1,v_2\in W(u):=\{v\in{\mathbb{X}_0^{s,p}(\Omega)}:||v-u||_{\mathbb{X}_0^{s,p}(\Omega)}<\gamma(u)\}$ we have 
\begin{equation}\label{4.50}
    ||A(v_1)-A(v_2)||_{\mathbb{X}_0^{s,p}(\Omega)}<\min \{\chi_1(v_1),\chi_1(v_2),\chi_2(v_1),\chi_2(v_2)\}.
    \end{equation}
Let $\mathcal{V}$ be a locally finite open refinement of $\{W(u):u\in E\}$. Define,
$$\mathcal{V^*}:=\{V\in \mathcal{V}:V\cap \bar{P}_\epsilon^+\neq\emptyset,V\cap \bar{P}_\epsilon^-\neq\emptyset,V\cap \bar{P}_\epsilon^+\cap \bar{P}_\epsilon^-=\emptyset\}$$
and 
$$\mathcal{U}:=\bigcup_{V\in \mathcal{V}\setminus \mathcal{V^*}}\{V\}\cup \bigcup_{V\in \mathcal{V^*}}\{V\setminus \bar{P}^+_\epsilon,V\setminus \bar{P}^-_\epsilon\}.$$
Note that $\mathcal{U}$ is a locally finite open refinement of $\{W(u):u\in E\}$ and it satisfies the property:
\begin{align}\label{4.51}
   &\text{For any} ~U\in \mathcal{U},~\text{such that}~U\cap \bar{P}^+_\epsilon\neq \emptyset~\text{and}~U\cap\bar{P}^-_\epsilon\neq \emptyset,~\text{we have}~U\cap \bar{P}_\epsilon^+\cap \bar{P}_\epsilon^-\neq\emptyset.
\end{align}
Let $\{\pi_U:U\in \mathcal{U}\}$ be the partition of unity subordinated to $\mathcal{U}$ defined by
$$\pi_U(u):=\left(\sum_{V\in\mathcal{U}}\alpha_V(u)\right)^{-1}\alpha_U(u),$$
where $$\alpha_U(u):=\dist(u,E\setminus U).$$
Using \eqref{4.51}, for any $U\in \mathcal{U}$, we choose $a_U\in U$ such that whenever $U\cap\bar{P}^\pm_\epsilon\neq \emptyset$, we have $a_U\in U\cap\bar{P}^\pm_\epsilon.$ We are now ready to define a map $B:E\rightarrow {\mathbb{X}_0^{s,p}(\Omega)}$ as 
$$B(u):=\sum_{U\in \mathcal{U}}\pi_U(u)A(a_U).$$
Therefore, the results $(a),(b),(c)$ follow from \eqref{4.48}, \eqref{4.49}, \eqref{4.50} and the Lemma \ref{lmn4.2}, Lemma \ref{lmn4.3}, Lemma \ref{lmn4.6}. If $f$ is odd, then $F(x,u)$ is even, which implies that the energy functional $\Phi$ is even and the map $A$ is odd. Now, we define 
$$\Tilde{B}(u):=\frac{1}{2}\left(B(u)-B(-u)\right).$$
Thus, the results $(a),(b),(c),(d)$ are straightforward from the definition of $\Tilde{B}$. Similarly, the results for $p>2$ can be obtained by using a similar argument as above and with the help of \cite[Lemma 4.2]{BL2004}. This completes the proof. 
\end{proof}
\begin{lemma}\label{lmn4.8}
 The functional $\Phi$ satisfies the $(PS)_c$-condition for any $c\in \mathbb{R}$.
\end{lemma}
\begin{proof}
    Let $(u_n)\subset \mathbb{X}_0^{s,p}(\Omega)$ be a sequence such that 
    \begin{equation}\label{eqqq}
        \Phi(u_n)\rightarrow c~\text{and}~ \Phi'(u_n)\rightarrow 0~\text{as}~n\rightarrow \infty.
    \end{equation}
    From the condition $(f_3)$, we get
    \begin{align}\label{eq4.45}
        \Phi(u_n)-\frac{1}{\mu} \langle \Phi'(u_n),u_n \rangle = & \left( \frac{1}{p}-\frac{1}{\mu}\right) \bigg( \int_\Omega |\nabla u_n|^{p}\d x +{\int_{\mathbb{R}^N} \int_{\mathbb{R}^N}} \frac{|u_n(x)-u_n(y)|^{p}}{|x-y|^{N+sp}}\d x \d y\bigg)\nonumber \\
        & - \int_\Omega\bigg(F(x,u_n)-\frac{1}{\mu}f(x,u_n)u_n \bigg)\d x.
    \end{align}
     On using \eqref{eqqq} and \eqref{eq4.45}, we have
    \begin{align}\label{eq4.46}
         \|u_n\|_{\mathbb{X}_0^{s,p}(\Omega)}^p \leq \frac{cp\mu}{\mu-p}~\text{as}~ n\rightarrow \infty.
    \end{align}
    Therefore, $(u_n)$ is bounded in $\mathbb{X}_0^{s,p}(\Omega)$. We define
    \begin{align*}
         I(u_n)& =\frac{1}{p}\bigg(\int_\Omega |\nabla u_n|^{p}\d x +{\int_{\mathbb{R}^N} \int_{\mathbb{R}^N}} \frac{|u_n(x)-u_n(y)|^{p}}{|x-y|^{N+sp}}\d x \d y\bigg)
         \end{align*}
         and
         \begin{align*}
         K(u_n)=\int_\Omega F(x,u_n) \d x.
    \end{align*} 
Therefore, we have 
\begin{equation*}
    \Phi(u_n)=I(u_n)-K(u_n)~\text{and}~\langle K'(u_n),\phi\rangle=\int_\Omega f(x,u_n)\phi \d x,~\forall \,\phi\in \mathbb{X}_0^{s,p}(\Omega).
\end{equation*}
 Recall Lemma \ref{inqlmn} when $p>2$ and define $\bar{u}_i(x,y)=u_i(x)-u_i(y)$ for $i=m,n$. Thus, we have 
    \begin{align}\label{eq4.47}
        \langle& I'(u_n)-I'(u_m),u_n-u_m \rangle = \langle I'(u_n),u_n-u_m \rangle-\langle I'(u_m),u_n-u_m \rangle \nonumber \\
        =&\int_\Omega \{|\nabla u_n|^{p-2} \nabla (u_n)-|\nabla u_m|^{p-2} \nabla (u_m)\}\cdot\nabla (u_n-u_m) \d x \nonumber\\
    &+{\int_{\mathbb{R}^N} \int_{\mathbb{R}^N}} \frac{\splitfrac{\{|\bar{u}_n(x,y)|^{p-2}(\bar{u}_n(x,y))-|\bar{u}_m(x,y)|^{p-2}(\bar{u}_m(x,y))\}}{\times((\bar{u}_n(x,y)-\bar{u}_m(x,y))}}{|x-y|^{N+sp}} \d x \d y\nonumber \\
    \geq &C\bigg( \int_\Omega |\nabla( u_n-u_m)|^{p} \d x +{\int_{\mathbb{R}^N} \int_{\mathbb{R}^N}} \frac{|\bar{u}_n(x,y)-\bar{u}_m(x,y)|^{p}}{|x-y|^{N+sp}}\d x \d y\bigg)\nonumber\\
    =&C\|u_n-u_m\|_{\mathbb{X}_0^{s,p}(\Omega)}^p.
    \end{align}
    Therefore, using \eqref{eq4.47}, we obtain
    \begin{align}\label{eq4.48}
        \|u_n-u_m\|_{\mathbb{X}_0^{s,p}(\Omega)}^{p-1}\leq C \| I'(u_n)-I'(u_m)\|_{{\mathbb{X}_0^{s,p}(\Omega)}^*}.
    \end{align}
   On the other hand, when $1<p\leq 2$, we use Lemma \ref{inqlmn} and reverse H\"older inequity to derive
    \begin{align}\label{eq4.49}
        \langle I'(u_n)&-I'(u_m),u_n-u_m \rangle = \langle I'(u_n),u_n-u_m \rangle-\langle I'(u_m),u_n-u_m \rangle \nonumber \\
        = &\int_\Omega \{|\nabla u_n|^{p-2} \nabla (u_n)-|\nabla u_m|^{p-2} \nabla (u_m)\}\cdot\nabla (u_n-u_m) \d x \nonumber\\
    &+{\int_{\mathbb{R}^N} \int_{\mathbb{R}^N}} \frac{\splitfrac{\{|\bar{u}_n(x,y)|^{p-2}(\bar{u}_n(x,y))-|\bar{u}_m(x,y)|^{p-2}(\bar{u}_m(x,y))\}}{\times((\bar{u}_n(x,y)-\bar{u}_m(x,y))}}{|x-y|^{N+sp}} \d x \d y\nonumber \\
    \geq & C \bigg[ \int_\Omega |\nabla u_n-\nabla u_m|^2 (|\nabla u_n|+|\nabla u_m|)^{p-2}\d x \nonumber \\
         &+ {\int_{\mathbb{R}^N} \int_{\mathbb{R}^N}} \frac{|\bar{u}_n(x,y)-\bar{u}_m(x,y)|^2(|\bar{u}_n(x,y)|+|\bar{u}_m(x,y)|)^{p-2}}{|x-y|^{N+sp}}\d x \d y\bigg]\nonumber \\
    \geq & C \bigg[ \left(\int_\Omega |\nabla u_n-\nabla u_m|^p \d x\right)^\frac{2}{p} \left(\int_\Omega (|\nabla u_n|+|\nabla u_m|)^p \d x\right)^\frac{p-2}{p} \nonumber\\
        &+\bigg\{\left( {\int_{\mathbb{R}^N} \int_{\mathbb{R}^N}}\frac{|\bar{u}_n(x,y)-\bar{u}_m(x,y)|^p}{|x-y|^{N+sp}}\d x\d y\right)^\frac{2}{p}\nonumber\\
        &\hspace{0.5cm} \times \left({\int_{\mathbb{R}^N} \int_{\mathbb{R}^N}} \frac{(|\bar{u}_n(x,y)|+|\bar{u}_m(x,y)|)^{p}}{|x-y|^{N+sp}}\d x \d y\right)^\frac{p-2}{p}\bigg\} \bigg]\nonumber \\
        = & C\bigg[\|\nabla(u_n-u_m)\|_p^2\left(\int_\Omega (|\nabla u_n|+|\nabla u_m|)^p\d x \right)^\frac{p-2}{p}\nonumber\\
        &+[u_m-u_n]_{s,p}^2\left({\int_{\mathbb{R}^N} \int_{\mathbb{R}^N}} \frac{(|\bar{u}_n(x,y)|+|\bar{u}_m(x,y)|)^{p}}{|x-y|^{N+sp}}\d x\d y\right)^\frac{p-2}{p}\bigg] \nonumber \\
        \geq &  C \|u_n-u_m\|_{\mathbb{X}_0^{s,p}(\Omega)}^2. 
    \end{align}
     From \eqref{eq4.49}, we conclude that
    \begin{align}\label{eq4.50}
        \|u_n-u_m\|_{\mathbb{X}_0^{s,p}(\Omega)}^{}\leq C \| I'(u_n)-I'(u_m)\|_{{\mathbb{X}_0^{s,p}(\Omega)}^*}.
    \end{align}
    On using the definitions of $I$ and $K$, we have
    \begin{align}\label{eq4.51}
        \| I'(u_n)-I'(u_m)\|_{{\mathbb{X}_0^{s,p}(\Omega)}^*}&=\|\Phi'(u_n)-\Phi'(u_m)+K'(u_n)-K'(u_m)\|_{{\mathbb{X}_0^{s,p}(\Omega)}^*}\nonumber \\
        & \leq \|\Phi'(u_n)-\Phi'(u_m)\|_{{\mathbb{X}_0^{s,p}(\Omega)}^*}+\|K'(u_n)-K'(u_m)\|_{{\mathbb{X}_0^{s,p}(\Omega)}^*}.
    \end{align}
    Since, $K':{\mathbb{X}_0^{s,p}(\Omega)}\rightarrow {{\mathbb{X}_0^{s,p}(\Omega)}^*} $ is compact, we conclude that $\{K'(u_n)\}$ possesses a convergent subsequence. Therefore, using \eqref{eq4.48} and \eqref{eq4.51} for $p>2$ and using \eqref{eq4.50} and \eqref{eq4.51} for $1<p\leq 2$, we conclude that $(u_n)$ has a convergent subsequence. This completes the proof.
\end{proof}

\begin{lemma}\label{lmn4.9}
    Let $M_c\setminus W=\emptyset$. Then there exists $\epsilon_0>0$ such that, for every $0<\epsilon<\epsilon'<\epsilon_0$, there exists a continuous mapping $\sigma:\mathbb{X}_0^{s,p}(\Omega)\rightarrow \mathbb{X}_0^{s,p}(\Omega)$ satisfying
    \begin{itemize}
    \item[$(a)$] $\sigma(0,u)=u,~\forall\,u\in \mathbb{X}_0^{s,p}(\Omega),$
    \item[$(b)$] $\sigma(t,u)=u,~\forall\,t\in[0,1],~ u\notin \Phi^{-1}[c-\epsilon',c+\epsilon'],$
    \item[$(c)$] $\sigma(1,\Phi^{c+\epsilon}\setminus W)\subset \Phi^{c-\epsilon}$,
    \item[$(d)$] $\sigma(t,\bar{P}^-_\epsilon)\subset \bar{P}^-_\epsilon$ and $\sigma(t,\bar{P}^+_\epsilon)\subset \bar{P}^+_\epsilon$, $\forall\,t\in[0,1].$ 
    \end{itemize}
\end{lemma}
\begin{proof}
    Let $\delta>0$ be sufficiently small. Therefore, there exists 
    $$N_\delta=\{u\in{\mathbb{X}_0^{s,p}(\Omega)}:\dist(u,M_c)<\delta\}$$ such that $N_\delta\subset W$. From Lemma \ref{lmn4.8}, there exists $\epsilon_0,\rho_1>0$ such that
    \begin{align}\label{eq4.52}
        \|\Phi'(u)\|_{{\mathbb{X}_0^{s,p}(\Omega)}}\geq \rho_1,~\text{for}~u\in \Phi^{-1}([c-\epsilon_0,c+\epsilon_0])\setminus N_{\frac{\delta}{2}}.
    \end{align}
Again, from Lemma \ref{lmn4.5} and Lemma \ref{lmn4.7}-(b), there exists $b_1$ such that 
\begin{align}\label{eq4.53}
    \|u-B(u)\|_{\mathbb{X}_0^{s,p}(\Omega)}\geq\frac{1}{2}\|u-A(u)\|_{\mathbb{X}_0^{s,p}(\Omega)}\geq b_1,~\forall\,u\in \Phi^{-1}([c-\epsilon_0,c+\epsilon_0])\setminus N_{\frac{\delta}{2}}.
\end{align}
Without loss of generality, we choose, $\epsilon_0<\frac{a_0\delta}{4}$, where $a_0<\min\{T_1,T_2\}$ with 
\begin{equation*}
    T_1=\dfrac{a_12^pb_1^{p-1}}{8}~\text{and}~T_2=\dfrac{a_2b_1}{2a_5^{2-p}(1+2^{2-p}b_1^{2-p})}.
\end{equation*}
We now define locally Lipschitz continuous functions $k_1$ and $k_2$ with $0\leq k_1(u), k_2(u) \leq 1$ such that
\begin{align}\label{eq4.54}
   k_1(u)&:= \begin{cases}
        0,&~\text{if}~u\in N_\frac{\delta}{4},\\
        1,&~\text{if}~u\in  \mathbb{X}_0^{s,p}(\Omega)\setminus N_\frac{\delta}{2},
   \end{cases}\\
   k_2(u)&:= \begin{cases}\label{eq4.55}
        0,&~\text{if}~u\in {\mathbb{X}_0^{s,p}(\Omega)}\setminus \Phi^{-1}([c-\epsilon',c+\epsilon']),\\
        1,&~\text{if}~u\in \Phi^{-1}([c-\epsilon,c+\epsilon]),
   \end{cases}
\end{align}
where $0<\epsilon<\epsilon'<\epsilon_0$ and we define $S$ as
\begin{align}\label{eq4.56}
    S(u):=-\frac{u-B(u)}{\|u-B(u)\|_{\mathbb{X}_0^{s,p}(\Omega)}},~\forall\, u\in E.
\end{align}
For $u\in E$, consider the initial value problem;
\begin{align}\label{eq4.57}
    \begin{cases}
       & \frac{d\phi}{dt}=k_1(\phi)k_2(\phi)S(\phi),\\
       & \phi(0)=u.
    \end{cases}
\end{align}
The classical theory of ordinary differential equations, the problem \eqref{eq4.57} has a unique solution. We denote $\phi(t,u)$ as the solution with a maximal interval of existence of $[0,\infty)$. Define,
$$\sigma(t,u):=\phi(\frac{2\epsilon}{a_0}t,u).$$
Clearly, $\sigma(t,0)=\phi(0,u)=u$, proving $(a)$. Moreover, for $t\in[0,1]$ and $u\notin \Phi^{-1}[c-\epsilon',c+\epsilon']$, we have $k_2(u)=0$. Therefore, we get $\sigma(t,u)=u,~\forall\,t\in[0,1],~ u\notin \Phi^{-1}[c-\epsilon',c+\epsilon']$. This proves $(b)$. We prove $(c)$ by a method of contradiction. Observe that if $u\in \Phi^{c+\epsilon}\setminus W$, then $u\notin N_\delta$. Suppose, that 
\begin{equation}\label{eqcont}
    \Phi(\phi(t,u))\geq c-\epsilon,~\forall\, t\in \bigg[0,\frac{2\epsilon}{a_0}\bigg].
\end{equation}
Clearly, $k_2(\phi(t,u))=1$ for all $t\in \bigg[0,\frac{2\epsilon}{a_0}\bigg]$. Moreover, for $t\in \bigg[0,\frac{2\epsilon}{a_0}\bigg]$, we have
\begin{align*}
    \|\phi(t,u)-u\|_{\mathbb{X}_0^{s,p}(\Omega)}&=\left|\left|\int_0^t\frac{d\phi}{ds}ds\right|\right|_{\mathbb{X}_0^{s,p}(\Omega)}\leq \int_0^t\left|\left|\frac{d\phi}{ds}\right|\right|_{\mathbb{X}_0^{s,p}(\Omega)}ds\leq t\leq \frac{2\epsilon}{a_0}< \frac{2\epsilon_0}{a_0}<\frac{\delta}{2}.
\end{align*}
Also, for $v\in M_c$, we get
\begin{align*}
\|\phi(t,u)-v\|_{\mathbb{X}_0^{s,p}(\Omega)}&=\|\phi(t,u)-u+u-v\|_{\mathbb{X}_0^{s,p}(\Omega)}\\
&\geq\|u-v\|_{\mathbb{X}_0^{s,p}(\Omega)}-\|\phi(t,u)-u\|_{\mathbb{X}_0^{s,p}(\Omega)}\\
& >\delta-\frac{\delta}{2}=\frac{\delta}{2}.
\end{align*}
Thus $\phi(t,u)\in \mathbb{X}_0^{s,p}(\Omega)\setminus N_\frac{\delta}{2}$, for all $t\in [0,\frac{2\epsilon}{a_0}]$ and hence we have $k_1(\phi(t,u))=1,~\forall \,t\in [0,\frac{2\epsilon}{a_0}]$. On applying Lemma \ref{lmn4.7}-$(b)$, $(c)$, and \eqref{eq4.53} for $p>2$, we get
\begin{align}\label{eq4.58}
   \Phi\bigg(\phi\bigg(\frac{2\epsilon}{a_0},u\bigg)\bigg)& = \Phi(u)+\int_0^{\frac{2\epsilon}{a_0}}\langle \Phi'(\phi(s,u)),S(\phi(s,u)) \rangle \d s \nonumber\\
& = \Phi(u)-\int_0^{\frac{2\epsilon}{a_0}}\bigg \langle \Phi'(\phi(s,u)), \frac{\phi(s,u)-B(\phi(s,u))}{\|\phi(s,u)-B(\phi(s,u))\|_{\mathbb{X}_0^{s,p}(\Omega)}}\bigg \rangle\d s \nonumber\\
& \leq \Phi(u)-\int_0^{\frac{2\epsilon}{a_0}}\frac{\langle \Phi'(\phi(s,u)), {\phi(s,u)-B(\phi(s,u))}\rangle}{2\|\phi(s,u)-A(\phi(s,u))\|_{\mathbb{X}_0^{s,p}(\Omega)}} \d s\nonumber\\
& \leq \Phi(u)-\frac{a_1}{4}  \int_0^{\frac{2\epsilon}{a_0}} \|\phi(s,u)-A(\phi(s,u))\|_{\mathbb{X}_0^{s,p}(\Omega)}^{p-1} \d s \nonumber \\
& \leq \Phi(u)-\frac{a_1 2^{p-1}b_1^{p-1}2\epsilon}{4a_0} \nonumber\\
&=\Phi(u)-\frac{T_12\epsilon}{a_0}\nonumber\\
&< c+\epsilon -2\epsilon=c-\epsilon.
\end{align}
This contradicts \eqref{eqcont}. To prove the result for $1<p\leq 2$, we use the following two properties: 
\begin{itemize}
    \item[$(i)$]  $(1+x)^r\leq 1+x^r$, $\forall\, x\geq 0$ and $\forall\, r\in[0,1)$ and
    \item[$(ii)$]  $\tau(x)=\frac{x}{1+x^r}$ with $r\in[0,1)$ is strictly increasing function $\forall\,x\geq0$.
\end{itemize}
Now, on using Lemma \ref{lmn4.7}-$(b)$, $(c)$, Lemma \ref{lmn4.4} and \eqref{eq4.53} for the range $1<p\leq 2$, we get
\begin{align}\label{eq4.59}
     \Phi\bigg(\phi\bigg(\frac{2\epsilon}{a_0},u\bigg)\bigg)& = \Phi(u)-\int_0^{\frac{2\epsilon}{a_0}}\bigg \langle \Phi'(\phi(s,u)), \frac{\phi(s,u)-B(\phi(s,u))}{\|\phi(s,u)-B(\phi(s,u))\|_{\mathbb{X}_0^{s,p}(\Omega)}}\bigg \rangle\d s \nonumber\\
     & \leq \Phi(u)-\int_0^{\frac{2\epsilon}{a_0}}\frac{\langle \Phi'(\phi(s,u)), {\phi(s,u)-B(\phi(s,u))}\rangle}{2\|\phi(s,u)-A(\phi(s,u))\|_{\mathbb{X}_0^{s,p}(\Omega)}} \d s\nonumber\\
     & \leq \Phi(u)-\frac{a_2}{4} \int_0^{\frac{2\epsilon}{a_0}} \frac{\|\phi(s,u)-A(\phi(s,u))\|_{\mathbb{X}_0^{s,p}(\Omega)}}{(\|\phi(s,u)\|_{{\mathbb{X}_0^{s,p}(\Omega)}}+\|A(\phi(s,u))\|_{\mathbb{X}_0^{s,p}(\Omega)})^{2-p}} \d s \nonumber \\
     & \leq \Phi(u)-\frac{a_2}{4a_5^{2-p}} \int_0^{\frac{2\epsilon}{a_0}} \frac{\|\phi(s,u)-A(\phi(s,u))\|_{\mathbb{X}_0^{s,p}(\Omega)}}{(1+\|\phi(s,u)-A(\phi(s,u))\|_{\mathbb{X}_0^{s,p}(\Omega)})^{2-p}} \d s \nonumber \\
     & \leq \Phi(u)-\frac{a_2}{4a_5^{2-p}} \int_0^{\frac{2\epsilon}{a_0}} \frac{\|\phi(s,u)-A(\phi(s,u))\|_{\mathbb{X}_0^{s,p}(\Omega)}}{1+\|\phi(s,u)-A(\phi(s,u))\|_{\mathbb{X}_0^{s,p}(\Omega)}^{2-p}} \d s \nonumber \\
     & \leq \Phi(u)-\frac{a_22b_1}{4a_5^{2-p}(1+2^{2-p}b_1^{2-p})} \int_0^{\frac{2\epsilon}{a_0}} \d s\nonumber\\
     & \leq c+\epsilon -T_2\frac{2\epsilon}{a_0}\nonumber\\
     &<c+\epsilon-2\epsilon=c-\epsilon,
\end{align}
which gives a contradiction to \eqref{eqcont}. This proves $(c)$. Since, $\phi(t,u)=u+t\frac{d}{dt}\phi(0,u)+o(t)$ as $t\rightarrow 0$ and $P^+_\epsilon,P^-_\epsilon$ are convex sets. Therefore, from Lemma \ref{lmn4.7}-$(a)$ and the convexity of $P^\pm_\epsilon$ conclude $(d)$, that is,  $\sigma(t,\bar{P}^-_\epsilon)\subset \bar{P}^-_\epsilon$ and $\sigma(t,\bar{P}^+_\epsilon)\subset \bar{P}^+_\epsilon$, $\forall\,t\in[0,1].$ This completes the proof.
\end{proof}
\noindent We now present a direct conclusion combining Lemma \ref{lmn4.7}-$(d)$ and Lemma \ref{lmn4.9}.
\begin{lemma}\label{lmn4.10}
    Let $N_1\subset {\mathbb{X}_0^{s,p}(\Omega)}$ be a symmetric neighborhood of $M_c^*=\{u\in M_c:u\neq0\}$. Then there exists $\epsilon_0>0$ with $0<\epsilon<\epsilon'<\epsilon_0$ such that there exists an odd and continuous map $\sigma:\mathbb{X}_0^{s,p}(\Omega)\rightarrow \mathbb{X}_0^{s,p}(\Omega)$ such that
    \begin{itemize}
        \item[$(a)$] $\sigma(0,u)=u,~\forall\,u\in \mathbb{X}_0^{s,p}(\Omega),$
        \item[$(b)$] $\sigma(t,u)=u,~\forall\,t\in[0,1],~ u\notin \Phi^{-1}[c-\epsilon',c+\epsilon'],$
        \item[$(c)$] $\sigma(1,\Phi^{c+\epsilon}\setminus N)\subset \Phi^{c-\epsilon}$,
        \item[$(d)$] $\sigma(t,\bar{P}^-_\epsilon)\subset \bar{P}^-_\epsilon$ and $\sigma(t,\bar{P}^+_\epsilon)\subset \bar{P}^+_\epsilon$, $\forall\,t\in[0,1].$ 
    \end{itemize}
\end{lemma}
\noindent We have now established the necessary results to prove our first main result.
   \section*{{\bf\emph{Proof of the Theorem \ref{T2.3}}}}
\noindent The proof is divided into two parts: the first part demonstrates the existence, while the second part proves the existence of infinitely many solutions.
\subsection*{Part I: Existence of solution:}
 In this part, we obtain the assumptions as in Theorem \ref{T3.1}. On using $(f_1), (f_2), (f_3)$ and the Sobolev inequality, we get for any $\epsilon>0$, there exists $C_\epsilon>0$ such that
\begin{align}
    \Phi(u)&=\frac{1}{p} \int_\Omega|\nabla u|^p \d x+\frac{1}{p} {\int_{\mathbb{R}^N} \int_{\mathbb{R}^N}} \frac{|u(x)-u(y)|^p}{|x-y|^{N+sp}}\d x\d y -\int_\Omega F(x,u) \d x \nonumber\\
    &\geq \frac{1}{p} \bigg[ \int_\Omega|\nabla u|^p \d x+ {\int_{\mathbb{R}^N} \int_{\mathbb{R}^N}} \frac{|u(x)-u(y)|^p}{|x-y|^{N+sp}}\d x\d y\bigg] -\epsilon \int_\Omega |u|^p \d x - C_\epsilon \int_\Omega |u|^q \d x\nonumber \\
    &\geq \bigg(\frac{1}{p}-\epsilon C_1-C_\epsilon C_2 \bigg)\bigg[ \int_\Omega|\nabla u|^p \d x+ {\int_{\mathbb{R}^N} \int_{\mathbb{R}^N}} \frac{|u(x)-u(y)|^p}{|x-y|^{N+sp}}\d x\d y\bigg] \nonumber
\end{align}
for some $C_1,C_2>0$ and $q\in(p,p^*)$. Therefore, there exists $\epsilon_0>0$ (sufficiently small), so that for all $\epsilon\in(0,\epsilon_0]$, there exists $\delta=\delta(\epsilon)>0$, such that
\begin{align}
    \Phi(u)&\geq \delta,~\forall\, u\in \mathbb{X}_0^{s,p}(\Omega)~\text{with}~ \|u\|_{\mathbb{X}_0^{s,p}(\Omega)}=\epsilon~\text{and}~\label{eq4.60}\\
    \Phi(u)&\geq 0,\forall u\in \mathbb{X}_0^{s,p}(\Omega)~\text{with}~ \|u\|_{\mathbb{X}_0^{s,p}(\Omega)}\leq\epsilon_0. \label{eq4.61}
\end{align}
Thus, we get 
\begin{align}\label{eq4.62}
    \inf_{u\in\overline{P^+_\epsilon}\cap \overline{P^-_\epsilon}} \Phi(u)=0,~\forall\, \epsilon\in(0,\epsilon_0],
\end{align}
From \eqref{eq4.62}, we get $0\in (\overline{P^+_\epsilon}\cap \overline{P^-_\epsilon} )$ is the unique critical point of $\Phi$. Consider two functions $v_1$, $v_2$ $\in C_c^\infty(\mathbb{R})\setminus\{0\}$ $v_1\leq0$, $v_2\geq0$ such that 
$$ supp(v_1)\cap supp(v_2)=\emptyset.$$
We define, 
\begin{equation}\label{psi def}
    \psi(k,l):=R(kv_1+lv_2),~\forall\,(k,l)\in \Delta~\text{and}~R\in\mathbb{R^+}.
\end{equation}
Thus for any $(k,l)\in \Delta$, we have
$$\psi(0,l)=Rlv_2\in P^+_\epsilon~\text{and}~\psi(k,o)=Rkv_1\in P^-_\epsilon.$$
From the assumptions $(f_1), (f_2), (f_3)$, there exist $C_3,C_4>0$ such that 
$$F(x,t)\geq C_3|t|^\mu-C_4,~\text{uniformly for}~x\in \bar{\Omega}.$$
Therefore, for any $u=R(tv_1+(1-t)v_2)\in \psi(\partial_0 \Delta)$,  $0\leq t\leq 1$, we get
\begin{align*}
    &\Phi(u) =\frac{1}{p} \int_\Omega|\nabla u|^p \d x+\frac{1}{p} {\int_{\mathbb{R}^N} \int_{\mathbb{R}^N}} \frac{|u(x)-u(y)|^p}{|x-y|^{N+sp}}\d x\d y -\int_\Omega F(x,u) \d x \nonumber\\
   \leq &  \frac{1}{p} \bigg[ \int_\Omega|\nabla u|^p \d x+ {\int_{\mathbb{R}^N} \int_{\mathbb{R}^N}}\frac{|u(x)-u(y)|^p}{|x-y|^{N+sp}}\d x\d y\bigg] -C_3 \int_\Omega |u|^\mu \d x\nonumber+C_4|\Omega| \nonumber \\
   = & \bigg[R^p \frac{1}{p} \|tv_1+(1-t)v_2\|_{\mathbb{X}_0^{s,p}(\Omega)} -C_3 R^{\mu}\int_{{supp(v_1)\cap supp(v_2)}} |tv_1(x)+(1-t)v_2(x)|^\mu \d x\bigg]\nonumber \\
   &\hspace{12cm}+C_4|\Omega| \nonumber \\
     =&\bigg[  \frac{R^p}{p} \|tv_1+(1-t)v_2\|_{\mathbb{X}_0^{s,p}(\Omega)} -C_3 R^{\mu}\bigg(t^\mu\int_{supp(v_1)} |v_1|^\mu+(1- t)^\mu\int_{supp(v_2)} |v_2|^\mu\bigg)\bigg] \nonumber\\
     &\hspace{12cm}+C_4|\Omega|.
\end{align*}

\noindent Thus, we have $\Phi(R(tv_1+(1-t)v_2))\longrightarrow -\infty$ as $R\rightarrow \infty$. On choosing sufficiently large $R>0$ in \eqref{psi def} and setting $Z=P^+_\epsilon \cap P^-_\epsilon$, we obtain \begin{equation*}
    \sup_{u\in \psi(\partial_0\Delta)}\Phi(u)<0<c_*~\text{ and } \psi(\partial_0\Delta)\cap Z=\emptyset.
\end{equation*} 
Consequently, from \eqref{eq4.60},\eqref{eq4.62} and  Theorem \ref{T3.1}, we conclude that $\Phi$ possesses at least one critical point $u^* \in M_{c_0} \setminus W$, implying that $u^*$ is a sign-changing solution of \eqref{MP}.
\subsection*{Part II: Infinitely many solutions:}
 In this case, we first construct a sequence of functions $\{\psi_n\}$ similar to \eqref{psi def}. Let $\{v_{ij}\}\subset C_c^\infty(\mathbb{R}^N)\setminus{\{0\}}$ be  a collection of mutually disjoint functions for $1\leq i\leq n,~1\leq j\leq 2$, and let $v_j=(v_{1j},v_{2j},...,v_{nj})$, $j=1,2$. For any $t=(t_1,t_2)\in B_{2n}$, we define
$$\psi_n(t)=\psi_n(t_1,t_2)=R_n(t_1v_1+t_2v_2) \text{ for } t_1,t_2\in B_n, t_j=(t_{1,j},t_{2,j},...t_{n,j}), j=1,2,$$ 
where $R_n>0$ and $t_jv_j=t_{1j}v_{1j}+t_{2j}v_{2j}+....+t_{nj}v_{nj},~j=1,2$. By definition, we have $\psi_n\in C(B_{2n},{\mathbb{X}_0^{s,p}(\Omega)})$. Moreover, $\psi_n(0)=0\in P_\epsilon^+\cap P_\epsilon^-$ and $\psi_n(-t)=-\psi_n(t)$, $\forall\, t\in B_{2n}$. Therefore, proceeding with similar arguments as in \textit{Part I}, there exists $R_n>0$, large enough, such that 
$$\sup_{u\in \psi_n(\partial B_{2n})}\Phi(u)<0<\inf_{u\in \partial P^+_\epsilon \cap \partial P^-_\epsilon}\Phi(u).$$
From Theorem \ref{T3.2}, we get
\begin{align}\label{eq4.64}
    c_j=\inf_{B\in \Gamma_j}\sup_{u\in {B\setminus W}}\Phi(u),
\end{align}
where $\Gamma_j$ refers to the definition as in Theorem \ref{T3.2}. 
Therefore, from \eqref{eq4.64}, we deduce that $c_j$ is a critical value of $\Phi$, for all $j\geq 3$, and there exists a sequence $\{u_j\}_{j\geq3}\subset {\mathbb{X}_0^{s,p}(\Omega)} \setminus W$ such that $u_j\in M_{c_j}\setminus W$ and $\Phi(u_j)=c_j\longrightarrow +\infty$ as $j\rightarrow +\infty$. This completes the proof.

\section{Least Energy Sign-Changing Solution For The Mixed Operator}\label{sec5}

\begin{lemma}\label{lmn5.1}
    According to the assumptions of Theorem \ref{T2.4} and $q\in(p,p^*)$, there exist $\mu_1, \mu_2>0$ such that\\
    $(a)$ $\|u^{\pm}\|_{{\mathbb{X}_0^{s,p}(\Omega)}}\geq \mu_1$, $\forall\,u\in \mathcal{M}$,\\
    $(b)$ $\int_\Omega |u^{\pm}|^q \d x\geq \mu_2$, $\forall\,u\in \mathcal{M}$.
\end{lemma}
\begin{proof}
    Let $u\in \mathcal{M}$. Then we get 
    $$\langle \Phi'(u),u^{\pm} \rangle=0 \text{ and } \int_\Omega f(x,u)u^+\d x=\int_\Omega f(x,u^+)u^+\d x.$$
    Now, observe that
    \begin{align}\label{eq5.1}
        0=& \langle \Phi'(u),u^{+} \rangle=\int_\Omega |\nabla u|^{p-2} \nabla u \cdot\nabla u^+ \d x\nonumber\\
        & +{\int_{\mathbb{R}^N} \int_{\mathbb{R}^N}}\frac{|u(x)-u(y)|^{p-2}(u(x)-u(y))(u^+(x)-u^+(y))}{|x-y|^{N+sp}} \d x \d y - \int_\Omega f(x,u)u^+ \d x \nonumber\\
        =& \int_\Omega |\nabla u^+|^{p-2} \nabla u^+\cdot\nabla u^+ \d x +{\int_{\mathbb{R}^N} \int_{\mathbb{R}^N}}\frac{\splitfrac{|u^+(x)-u^+(y)|^{p-2}(u^+(x)-u^+(y))}{\times(u^+(x)-u^+(y))}}{|x-y|^{N+sp}} \d x \d y \nonumber\\
        & +2\bigg[\int_{\Omega^+}\int_{\Omega^-}\frac{|u^+(x)-u^-(y)|^{p-2}(u^+(x)-u^-(y))u^+(x)}{|x-y|^{N+sp}}\d x\d y \nonumber \\
        &-\int_{\Omega^+}\int_{\Omega^-}\frac{|u^+(x)|^{p}}{|x-y|^{N+sp}}\d x\bigg] -\int_\Omega f(x,u^+)u^+ \d x\nonumber\\
        =&  \langle \Phi'(u^+),u^{+} \rangle + 2E_1^+(u),
    \end{align}
    where $$E_1^+(u)=\int_{\Omega^+}\int_{\Omega^-}\frac{|u^+(x)-u^-(y)|^{p-1}u^+(x)}{|x-y|^{N+sp}}\d x\d y-\int_{\Omega^+}\int_{\Omega^-}\frac{|u^+(x)|^{p}}{|x-y|^{N+sp}}\d x.$$
    Clearly, $E_1^+(u)>0$. Thus from \eqref{eq5.1}, we get $\langle \Phi'(u^+),u^{+} \rangle<0$. Therefore, we have
    $$\|u^+\|_{\mathbb{X}_0^{s,p}(\Omega)}^p<\int_\Omega f(x,u^+)u^+ \d x.$$
    Similarly, we obtain 
    $$\|u^-\|_{\mathbb{X}_0^{s,p}(\Omega)}^p<\int_\Omega f(x,u^-)u^- \d x.$$
    Thanks to conditions $(f_1)$ and $(f_2)$, for any $\epsilon>0$, there exists $C_\epsilon>0$ such that
    \begin{align}\label{5.2}
        f(x,t)t\leq \epsilon |t|^p+C_\epsilon |t|^q,~\forall\,x\in \bar{\Omega} ~\text{and}~\forall\,t\in \mathbb{R}.
        \end{align}
   On using the mixed Sobolev inequality \eqref{eq2.8S} and H\"older inequality, we get
    \begin{align}\label{eq5.2}
        \|u^\pm\|_{\mathbb{X}_0^{s,p}(\Omega)}^p&<\int_\Omega f(x,u^\pm)u^\pm \d x\leq \int_\Omega(\epsilon |u^\pm|^p+C_\epsilon |u^\pm|^q) \d x\\
        & \leq \epsilon C_1 \|u^\pm\|^p_{\mathbb{X}_0^{s,p}(\Omega)}+ C_\epsilon C_2\|u^\pm\|^q_{\mathbb{X}_0^{s,p}(\Omega)},\label{eq5.3}
    \end{align}
    for some  $C_1, C_2>0$. Since, $q\in(p,p^*)$. Choose $\epsilon=\frac{1}{2C_1}$ in \eqref{eq5.3}, we deduce
    \begin{align}\label{eq5.4}
        C_\epsilon C_2\|u^\pm\|^q_{\mathbb{X}_0^{s,p}(\Omega)} \geq \frac{1}{2} \|u^\pm\|_{\mathbb{X}_0^{s,p}(\Omega)}^p.
    \end{align}
    Consequently, the inequality \eqref{eq5.4} gives
    \begin{align}\label{eq5.5}
         \|u^{\pm}\|_{{\mathbb{X}_0^{s,p}(\Omega)}}\geq \bigg(\frac{1}{2C_\epsilon C_2}\bigg)^\frac{1}{q-p}=\mu_1.
    \end{align}
    Again, using the H\"older inequality and \eqref{eq5.2}, we get 
    \begin{align}\label{eq5.6}
        \epsilon C_1 \|u^\pm\|^p_{\mathbb{X}_0^{s,p}(\Omega)}+ C_\epsilon|u^\pm|_q^q \geq \|u^\pm\|_{\mathbb{X}_0^{s,p}(\Omega)}^p. 
    \end{align}
     Finally choosing $\epsilon=\frac{1}{2C_1}$ in \eqref{eq5.6} and from \eqref{eq5.5}, we obtain
     \begin{align}\label{eq5.7}
         |u^\pm|_q^q \geq \frac{1}{2C_\epsilon}\|u^\pm\|_{\mathbb{X}_0^{s,p}(\Omega)}^p \geq \frac{\mu^p_1}{2C_\epsilon}=\mu_2.
     \end{align}
     This completes the proof.
\end{proof}
\begin{lemma}\label{lmn5.2}
    Suppose that $(f_1)$, $(f_2)$, $(f_4)$, $(f_5)$ hold, and $f\in C^1(\bar{\Omega}\cross\mathbb{R},\mathbb{R})$. Then for any $u\in {\mathbb{X}_0^{s,p}(\Omega)}\setminus \{0\}$, there exists a unique $\tau_0\in \mathbb{R^+}$ such that $\tau_0u\in \mathcal{N}$. Furthermore, for any $u\in \mathcal{N}$, we have 
    \begin{align}\label{eq5.8}
        \Phi(u)=\max_{t\in[0,\infty)} \Phi(tu).
    \end{align}
\end{lemma}
\begin{proof}
    For $u\in {\mathbb{X}_0^{s,p}(\Omega)}\setminus \{0\}$ and $t\geq0$, we define 
    $$h(t)=\Phi(tu)=\frac{t^p}{p}\|u\|^p_{\mathbb{X}_0^{s,p}(\Omega)}-\int_\Omega F(x,tu) \d x.$$ 
    Clearly, $h(0)=0$ and
    \begin{align}\label{eq5.9}
        h'(t)=\langle \Phi'(tu),u\rangle=t^{p-1}\|u\|_{\mathbb{X}_0^{s,p}(\Omega)}^p-\int_\Omega f(x,tu)u \d x.
    \end{align}
    Thus $h'(\tau_0)=0 $ for some $\tau_0$ if and only if $$\tau_0^p\|u\|_{\mathbb{X}_0^{s,p}(\Omega)}^p=\int_\Omega f(x,\tau_0u)\tau_0u \d x.$$
Therefore, $\tau_0u\in \mathcal{N}$. Using $(f_1)$, $(f_2)$ and $(f_4)$, there exists $\delta>0$ such that $h(t)>0$ if $t\in(0,\delta)$ and $h(t)<0$ if $t\in(\frac{1}{\delta},+\infty)$. Since $h(0)=0$, there exists $\tau_0>0$ such that $h$ has a global maximum at $\tau_0$. Hence $h'(\tau_0)=0$ and we get $\tau_0 u\in \mathcal{N}$. Moreover, using $(f_5)$, we see that $\frac{h'(t)}{t^{p-1}}$ is strictly monotone in $t\in(0,\infty)$. Thus, $\tau_0$ is unique such that $\tau_0 u\in \mathcal{N}$.
    \par For $u\in \mathcal{N}$, we have that $\tau_0=1$ and $h(t)$ increase in $(0,1)$ and decrease in $(1,\infty)$. Therefore, we obtain
    $$\Phi(u)=h(1)=\max_{t\in[0,1]}h(t)=\max_{t\in[0,\infty)}h(t).$$
     This completes the proof.
\end{proof}
\begin{lemma}\label{lmn5.3}
    There exists $u_*\in {\mathbb{X}_0^{s,p}(\Omega)}$ such that $c_s$ can be at $u_*$.
\end{lemma}
\begin{proof}
    Using the arguments in Theorem \ref{T2.3}, $\Phi(tu)$ has a strictly local minimum at $t=0$ and $\Phi(tu)\rightarrow -\infty$ as $t \rightarrow \infty$, $\forall$ $u\in \mathbb{X}_0^{s,p}(\Omega)\setminus\{0\}$. Furthermore, using the Lemma \ref{lmn5.2}, for any $u\in \mathbb{X}_0^{s,p}(\Omega)\setminus\{0\}$, there exists $\tau_u>0$ such that 
    $$\Phi(\tau_uu)=\max_{t\in(0,\infty)} \Phi(tu)>0.$$
    Then $$c:=\inf_{u\in \mathbb{X}_0^{s,p}(\Omega)\setminus\{0\}}\max_{t\in(0,\infty)} \Phi(tu)>0,$$
    is well-defined. Consider a minimizing sequence $(u_n)\subset \mathbb{X}_0^{s,p}(\Omega)$ such that
    \begin{align}\label{5.10}
        \Phi(u_n)=\max_{t\in(0,\infty)} \Phi(tu_n)\rightarrow c.
         \end{align}
    We first claim that $(u_n)$ is bounded. If not, then $\|u_n\|_{\mathbb{X}_0^{s,p}(\Omega)}\rightarrow \infty$ and set $w_n=\frac{u_n}{\|u_n\|_{\mathbb{X}_0^{s,p}(\Omega)}}.$ Thus $\|w_n\|_{\mathbb{X}_0^{s,p}(\Omega)}=1,$ that is, $(w_n)$ is a bounded sequence. Then there exists a subsequence, still denoted by $(w_n)$ such that $w_n \rightharpoonup  w\in \mathbb{X}_0^{s,p}(\Omega)$ weakly in $\mathbb{X}_0^{s,p}(\Omega)$. Then from compact embedding, we have $w_n\rightarrow w$ strongly in $L^r(\Omega)$ for $r\in[1,p^*)$ and $w_n(x)\rightarrow w(x)$ $a.e.$ in $\Omega$. Let $w\neq0$ and $\Omega_1=\{x\in \mathbb{X}_0^{s,p}(\Omega): w(x)\neq0 \}$. Then using $(f_4)$ and Fatou's lemma, we deduce that
    \begin{align}\label{eq5.10}
        \frac{1}{p} - \frac{c+o(1)}{\|u_n\|_{\mathbb{X}_0^{s,p}(\Omega)}^p}= \frac{1}{p} - \frac{\Phi(u_n)}{\|u_n\|_{\mathbb{X}_0^{s,p}(\Omega)}^p}=\int_\Omega \frac{F(x,u_n)}{u_n^p}w_n^p \d x \geq \int_{\Omega_1} \frac{F(x,u_n)}{u_n^p}w_n^p \d x.
    \end{align}
    Taking the limit on both sides of the equation \eqref{eq5.10} implies that
    \begin{align}
        \frac{1}{p}\geq \liminf_{n\rightarrow \infty} \int_{\Omega_1} \frac{F(x,u_n)}{u_n^p}w_n^p \d x \geq \int_{\Omega_1} \liminf_{n\rightarrow \infty} \frac{F(x,u_n)}{u_n^p}w_n^p \d x= \infty,
    \end{align}
    which is a contradiction. Again, if $w=0$ and fixed $R>(cp)^\frac{1}{p}$. By $(f_1)$ and $(f_2)$, we say that $\int_\Omega F(x,u) \d x$ is weakly continuous in $\mathbb{X}_0^{s,p}(\Omega)$. Then from \eqref{5.10}, we obtain
    \begin{align}\label{eq5.13}
        c+o(1)=\Phi(u_n)\geq \Phi(Rw_n)=\frac{1}{p}R^p-\int_\Omega F(x,Rw_n) \d x=\frac{1}{p}R^p+o(1),
    \end{align}
   again taking the limit on both sides of the equation \eqref{eq5.13}, we get
   $$c\geq\frac{1}{p}R^p,$$
   which is a contradiction. Thus, $(u_n)$ is bounded in $\mathbb{X}_0^{s,p}(\Omega)$. By similar arguments in \cite[Theorem 2.1]{LW2004}, there exists $u_*\in \mathcal{N}$ such that $c_s=\Phi(u_*)$. 
    This completes the proof.
\end{proof}
\begin{lemma}\label{lmn5.4}
    If $u\in \mathbb{X}_0^{s,p}(\Omega)$ with $u^{\pm}\neq0$, then there exists a unique pair $(k_u,l_u)$ of positive numbers such that 
    $$k_uu^++l_uu^-\in\mathcal{M}.$$
\end{lemma}
\begin{proof}
   \underline{\textbf{{First, we prove the existence:}}} For any $k,l>0$, let us consider the functions, $g_1$ and $g_2$, which are defined as
    \begin{align}\label{eq5.15}
    &g_1(k,l)=\langle \Phi'(ku^++lu^-), ku^+\rangle\nonumber\\
    =&\int_{\Omega^+} |\nabla ku^+|^{p} \d x+\int_{\Omega^+}\int_{\Omega^+} \frac{|ku^+(x)-ku^+(y)|^{p}}{|x-y|^{N+sp}} \d x \d y {+\int_{\Omega^+}\int_{\mathbb{R^N}\setminus \Omega} \frac{|ku^+(x)|^{p}}{|x-y|^{N+sp}} \d x \d y}\nonumber\\
        &+\int_{\Omega^+}\int_{\Omega^-}\frac{|ku^+(x)-lu^-(y)|^{p-1}ku^+(x)}{|x-y|^{N+sp}}\d x\d y {+\int_{\mathbb{R^N}\setminus \Omega} \int_{\Omega^+}\frac{|ku^+(y)|^{p}}{|x-y|^{N+sp}} \d x \d y }
        \nonumber\\
        &+\int_{\Omega^-}\int_{\Omega^+}  \frac{|lu^-(x)-ku^+(y)|^{p-1}ku^+(y)}{|x-y|^{N+sp}}\d x\d y - \int_\Omega f(x,ku^+)ku^+ \d x,
    \end{align}
    and \begin{align}\label{eq5.16}
        &g_2(k,l)=\langle \Phi'(ku^++lu^-), lu^-\rangle\nonumber\\  =&\int_{\Omega^-} |\nabla lu^-|^{p} \d x+\int_{\Omega^-}\int_{\Omega^-} \frac{|lu^-(x)-lu^-(y)|^{p}}{|x-y|^{N+sp}} \d x \d y{+\int_{\Omega^-}\int_{\mathbb{R^N}\setminus \Omega} \frac{|lu^-(x)|^{p}}{|x-y|^{N+sp}} \d x \d y} \nonumber\\
        & +\int_{\Omega^+}\int_{\Omega^-}\frac{|ku^+(x)-lu^-(y)|^{p-1}(-lu^-(y))}{|x-y|^{N+sp}}\d x\d y {+\int_{\mathbb{R^N}\setminus \Omega} \int_{\Omega^-}\frac{|-lu^-(y)|^{p}}{|x-y|^{N+sp}} \d x \d y}
        \nonumber\\
        &+\int_{\Omega^-}\int_{\Omega^+}  \frac{|lu^-(x)-ku^+(y)|^{p-1}(-lu^-(x))}{|x-y|^{N+sp}}\d x\d y - \int_\Omega f(x,lu^-)lu^- \d x.
    \end{align}
Using $(f_4)$, for $C_1>0$, there exists $C_2>0$ such that
\begin{align}\label{eq5.17}
    f(x,t)t\geq C_1|t|^p-C_2,~\forall\,x\in \bar{\Omega}~\text{and}~\forall\,t\in \mathbb{R}.
\end{align}
From \eqref{5.2}, \eqref{eq5.17} combined with Lemma \ref{lmn5.1} and $p<q<p^*$, there exists $r_1>0$ (small enough) and $R_1>0$ (large enough) such that 
\begin{align}\label{eq5.18}
   & g_1(k,k)>0,~g_2(k,k)>0,~\forall\,k\in(0,r_1),\\
    \text{and}~& g_1(k,k)<0,~g_2(k,k)<0,~\forall\,k\in(R_1,\infty).\label{eq5.19}
\end{align}
Observe that, for any fixed $k>0$, $g_1(k,l)$ is increasing in $(0,\infty)$ with respect to $l$ and for any fixed $l>0$, $g_2(k,l)$ is increasing in $(0,\infty)$ with respect to $k$. Thus, using \eqref{eq5.18} and \eqref{eq5.19}, there exist $r,R>0$ with $r<R$ such that 
\begin{align}\label{eq5.20}
   & g_1(r,l)>0,~g_1(R,l)<0,~\forall\,l\in(r,R],\\
    \text{and}~& g_2(k,r)>0,~g_2(k,R)<0,~\forall\,k\in(r,R].\label{eq5.21}
\end{align}
Now applying Miranda's theorem\cite{M1940}, there exist $k_u,l_u\in [r,R]$ such that $g_1(k_u,l_u)=0$ and $g_2(k_u,l_u)=0.$ Therefore, we get $k_uu^++l_uu^-\in \mathcal{M}$.\\
\underline{\textbf {We now prove the uniqueness:}} Let  $(k_1,l_1)$ and $(k_2,l_2)$ be two different positive pairs
such that $k_iu^++l_iu^-\in \mathcal{M}$, $i=1,2$. We divide the proof of uniqueness into two separate cases.
\par \textbf{Case I:} Let $u\in \mathcal{M}$. Without loss of generality, we may assume that $(k_1,l_1)=(1,1)$ and $k_2\leq l_2$. For each $u\in \mathbb{X}_0^{s,p}(\Omega)$, we define 
\begin{align}
    A^+(u)=&\int_{\Omega^+} |\nabla u^+|^{p} \d x +\int_{\Omega^+}\int_{\Omega^+} \frac{|u^+(x)-u^+(y)|^{p}}{|x-y|^{N+sp}} \d x \d y \nonumber\\
    &+\int_{\Omega^+}\int_{\Omega^-}\frac{|u^+(x)-u^-(y)|^{p-1}u^+(x)}{|x-y|^{N+sp}}\d x\d y+ {\int_{\Omega^+}\int_{\mathbb{R^N}\setminus \Omega} \frac{|u^+(x)|^{p}}{|x-y|^{N+sp}} \d x \d y} \nonumber \\
     &+\int_{\Omega^-}\int_{\Omega^+}\frac{|u^-(x)-u^+(y)|^{p-1}u^+(y)}{|x-y|^{N+sp}}\d x\d y+{\int_{\mathbb{R^N}\setminus \Omega} \int_{\Omega^+}\frac{|u^+(y)|^{p}}{|x-y|^{N+sp}} \d x \d y, } 
     \end{align}
     and
     \begin{align}
     A^-(u)=&\int_{\Omega^-} |\nabla u^-|^{p} \d x +\int_{\Omega^-}\int_{\Omega^-} \frac{|u^-(x)-u^-(y)|^{p}}{|x-y|^{N+sp}} \d x \d y \nonumber\\
    &+\int_{\Omega^+}\int_{\Omega^-}\frac{|u^+(x)-u^-(y)|^{p-1}(-u^-(y))}{|x-y|^{N+sp}}\d x\d y {+\int_{\Omega^-}\int_{\mathbb{R^N}\setminus \Omega} \frac{|u^-(x)|^{p}}{|x-y|^{N+sp}} \d x \d y}\nonumber \\
    &+\int_{\Omega^-}\int_{\Omega^+}\frac{|u^-(x)-u^+(y)|^{p-1}(-u^-(x))}{|x-y|^{N+sp}}\d x\d y+{\int_{\mathbb{R^N}\setminus \Omega} \int_{\Omega^-}\frac{|-u^-(y)|^{p}}{|x-y|^{N+sp}} \d x \d y }.\nonumber
    \end{align}
    Since, $u\in \mathcal{M}$, we have $\langle\Phi'(u),u^\pm \rangle=0$, Thus
    \begin{align}\label{eq5.22}
        A^+(u)=\int_\Omega f(x,u^+)u^+ \d x,\\
        ~\text{and}~
        A^-(u)=\int_\Omega f(x,u^-)u^- \d x.\label{eq5.23} 
    \end{align}
    Again using $\langle \Phi'(k_2u^++l_2u^-), k_2u^+\rangle=\langle \Phi'(k_2u^++l_2u^-), l_2u^-\rangle=0$, we obtain
    \begin{align}\label{eq5.24}
        k_2^p(A^+(u)+B_1^+(u)+B_2^+(u))= \int_\Omega f(x,k_2u^+)k_2u^+ \d x\\ ~\text{and}~
        l_2^p(A^-(u)+B_1^-(u)+B_2^-(u))=\int_\Omega f(x,l_2u^-)l_2u^- \d x, \label{eq5.25}
    \end{align}
    where
    \begin{align}\label{eq5.26}
        B_1^+(u)=&\int_{\Omega^+}\int_{\Omega^-}\frac{|u^+(x)-{l_2{k_2}^{-1}}u^-(y)|^{p-1}u^+(x)}{|x-y|^{N+sp}}\d x\d y\nonumber\\
        &-\int_{\Omega^+}\int_{\Omega^-}\frac{|u^+(x)-u^-(y)|^{p-1}u^+(x)}{|x-y|^{N+sp}}\d x\d y, \\ \label{eq5.27}
        B_2^+(u)=&\int_{\Omega^-}\int_{\Omega^+}\frac{|l_2{k_2}^{-1}u^-(x)-u^+(y)|^{p-1}u^+(y)}{|x-y|^{N+sp}}\d x\d y\nonumber \\
        &- \int_{\Omega^-}\int_{\Omega^+} \frac{|u^-(x)-u^+(y)|^{p-1}u^+(y)}{|x-y|^{N+sp}}\d x\d y, \\ \label{eq5.28}
        B_1^-(u)=&\int_{\Omega^+}\int_{\Omega^-}\frac{|k_2{l_2}^{-1}u^+(x)-u^-(y)|^{p-1}(-u^-(y))}{|x-y|^{N+sp}} \d x\d y\nonumber\\
        &- \int_{\Omega^+}\int_{\Omega^-}\frac{|u^+(x)-u^-(y)|^{p-1}(-u^-(y))}{|x-y|^{N+sp}}\d x\d y, \\ \label{eq5.29}
        B_2^-(u)=&\int_{\Omega^-}\int_{\Omega^+}\frac{|u^-(x)-k_2{l_2}^{-1}u^+(y)|^{p-1}(-u^-(x))}{|x-y|^{N+sp}} \d x\d y\nonumber\\
        &- \int_{\Omega^-}\int_{\Omega^+}\frac{|u^-(x)-u^+(y)|^{p-1}(-u^-(x))}{|x-y|^{N+sp}}\d x\d y.
    \end{align}
    Since $k_2\leq l_2$, we conclude that $B_1^+(u)$ and $B_2^+(u)\geq 0.$ Then using the equations \eqref{eq5.22} and \eqref{eq5.24}, we get
    \begin{align}\label{eq5.30}
        0\leq \int_\Omega \left[ \frac{f(x,k_2u^+)}{|k_2u^+|^{p-2}k_2u^+}-\frac{f(x,u^+)}{|u^+|^{p-2}u^+}\right]|u^+|^p \d x.
    \end{align}
    Therefore, $(f_5)$ asserts that $k_2\geq1$. On the other hand, $B_1^-(u)$ and $B_2^-(u)\leq0$. Thus using \eqref{eq5.23} and \eqref{eq5.25}, we obtain
    \begin{align}\label{eq5.31}
        0\geq \int_\Omega \left[ \frac{f(x,l_2u^-)}{|l_2u^-|^{p-2}l_2u^-}-\frac{f(x,u^-)}{|u^-|^{p-2}u^-}\right]|u^-|^p \d x,
    \end{align}
    which gives $l_2\leq1$ using $(f_5)$. Hence $k_2=l_2=1$.
    \par \textbf{Case II:} Let $u\notin \mathcal{M}$ and $v_1=k_1u^++l_1u^-$, $v_2=k_2u^++l_2u^-$. Proceeding with a similar arguments in \textbf{Case I}, we get $\frac{k_2}{k_1}=\frac{l_2}{l_1}=1$. Hence $(k_1,l_1)=(k_2,l_2)$. This completes the proof.
\end{proof}
\begin{lemma}\label{lmn5.5}
    According to the assumptions of Theorem \ref{T2.4}, there exists $u\in \mathcal{M}$ such that $\Phi(u)=m_s$, where $m_s=\inf_{u\in \mathcal{M}}\Phi(u)$. 
\end{lemma}
\begin{proof}
    Observe that $\mathcal{M}\neq \emptyset$, which makes the minimization problem $m_s=\inf_{u\in \mathcal{M}}\Phi(u)$ well-defined. Let $(u_n)\subset\mathcal{M}$ be a minimizing sequence such that $\Phi(u_n)\rightarrow m_s$ as $n\rightarrow \infty$. Proceeding with similar arguments in Lemma \ref{lmn5.3}, we get $(u_n)$ is uniformly bounded in $\mathbb{X}_0^{s,p}(\Omega)$. Thus, there exists a subsequence $(u_n)$ (still denoted by $(u_n)$) and $u^*\in \mathbb{X}_0^{s,p}(\Omega)$ such that 
    \begin{align}\label{eq5.32}
        u_n^\pm &\rightharpoonup (u^*)^\pm\in \mathbb{X}_0^{s,p}(\Omega)~\text{weakly in}~ \mathbb{X}_0^{s,p}(\Omega),\\\label{eq5.33}
        u_n^\pm &\rightarrow (u^*)^\pm~\text{strongly in}~ L^r(\Omega)~\text{for}~r\in[1,p^*),\\ \label{eq5.34}
        u_n(x)& \rightarrow u^*(x)~a.e.~\text{in}~\Omega.
    \end{align}
    By Lemma \ref{lmn5.1}, we have $(u^*)^\pm\neq0$. Moreover, by standard arguments (see \cite{W1996}), the conditions $(f_1), (f_2)$ and Theorem \ref{thm cpt} gives,
    \begin{align}\label{eq5.35}
        \lim_{n\rightarrow \infty} \int_\Omega f(x,u_n^\pm)u_n^\pm \d x& =\int_\Omega f(x,(u^*)^\pm)(u^*)^\pm \d x,\\
        \text{and}\nonumber\\\label{eq5.36}
        \lim_{n\rightarrow \infty} \int_\Omega F(x,u_n^\pm) \d x& =\int_\Omega F(x,(u^*)^\pm) \d x.
    \end{align}
    From Lemma \ref{lmn5.4}, there exist $k^*,l^*>0$ such that $k^*(u^*)^++l^*(u^*)^-\in \mathcal{M}$, which implies that
    \begin{align}\label{eq5.37}
        (k^*)^p(A^+(u^*)+B_1^+(u^*)+B_2^+(u^*))= \int_\Omega f(x,k^*(u^*)^+)k^*(u^*)^+ \d x,\\ \label{eq5.38}
        \text{and}~(l^*)^p(A^-(u^*)+B_1^-(u^*)+B_2^-(u^*))=\int_\Omega f(x,l^*(u^*)^-)l^*(u^*)^- \d x.
    \end{align}
    We first claim  that $k^*, l^*\leq 1$. Since, $(u_n)\subset\mathcal{M}$ is a minimizing sequence, we have $$\langle \Phi'(u_n),u_n^\pm \rangle=0,$$ that is
    \begin{align}\label{eq5.39}
        A^\pm(u_n)=\int_\Omega f(x,u_n^\pm)u^\pm \d x.
    \end{align}
    On using the inequalities \eqref{eq5.32}-\eqref{eq5.38} and Fatou lemma, deduce that
    \begin{align}\label{eq5.40}
        A^\pm(u^*)\leq \int_\Omega f(x,(u^*)^\pm)(u^*)^\pm \d x.
    \end{align}
    Now, without loss of generality, we take $k^*\leq l^*$. Since $B_1^-(u^*), B_2^-(u^*)\leq 0$, we obtain
    \begin{align}\label{eq5.41}
         0\leq \int_\Omega \left[ \frac{f(x,(u^*)^-)}{|(u^*)^-|^{p-2}(u^*)^-}-\frac{f(x,l^*(u^*)^-)}{|l^*(u^*)^-|^{p-2}l^*(u^*)^-}\right]|(u^*)^-|^p \d x.
    \end{align}
    Thus, using $(f_5)$ in \eqref{eq5.41}, we get $l^*\leq1$. Hence, $0<k^*\leq l^*\leq 1$.\\
    Next we prove that $k^*=1$ and  $l^*=1$, that is $u^*\in \mathcal{M}.$
    \par Let $\mathcal{H}(.,t)=f(.,t)t-pF(.,t)$. Then using $(f_5)$, we have 
$\mathcal{H}(.,t)$ is increasing in $t\in(0,+\infty)$, $\mathcal{H}(.,t)$ is decreasing in $t\in(-\infty, 0)$ and $\mathcal{H}(.,t)\geq 0$. Therefore, by Fatou lemma, we get
\begin{align}
    m_s &\leq \Phi(k^*(u^*)^++l^*(u^*)^-)\nonumber\\
    &=\Phi(k^*(u^*)^++l^*(u^*)^-)-\frac{1}{p}\langle \Phi'(k^*(u^*)^++l^*(u^*)^-),k^*(u^*)^++l^*(u^*)^-\rangle \nonumber\\
    &= \frac{1}{p} \int_\Omega \mathcal{H}(x,k^*(u^*)^++l^*(u^*)^-) \d x \nonumber \\
    &= \frac{1}{p}\bigg[\int_{\Omega^+} \mathcal{H}(x,k^*(u^*)^+ \d x +\int_{\Omega^-} \mathcal{H}(x,l^*(u^*)^-) \d x \bigg]\nonumber \\
    & \leq \frac{1}{p}\bigg[\int_{\Omega^+} \mathcal{H}(x,(u^*)^+ \d x +\int_{\Omega^-} \mathcal{H}(x,(u^*)^-) \d x \bigg]=\frac{1}{p}\int_{\Omega} \mathcal{H}(x,(u^*)^++(u^*)^-) \d x\nonumber \\
    &=\frac{1}{p}\int_{\Omega} \mathcal{H}(x,u^*) \d x\leq\frac{1}{p} \liminf_{n\rightarrow\infty}\int_\Omega \mathcal{H}(x,u_n) \d x =\lim_{n\rightarrow\infty} \bigg[\Phi(u_n)-\frac{1}{p}\langle\Phi'(u_n),u_n\rangle\bigg]\nonumber \\
    &=\lim_{n\rightarrow\infty} \Phi(u_n)=m_s.\nonumber
\end{align}
Hence, we conclude $k^*=l^*=1$ and $\Phi(u^*)=m_s$. This completes the proof.
\end{proof}
\begin{lemma}\label{lmn5.6}
    If $u\in \mathcal{M}$, then we have
    $$\Phi(u)>\Phi(k u^++lu^-),~\forall\,k,l\geq 0~\text{such that}~(k,l)\neq(1,1).$$
\end{lemma}
\begin{proof}
    For any $u\in \mathbb{X}_0^{s,p}(\Omega)$ such that $u^\pm\neq0$, we define $I_u:[0,\infty)\cross[0,\infty)\rightarrow \mathbb{R}$ as follows
    $$I_u(k,l)=\Phi(k u^++lu^-),~\forall\,k,l\geq0.~$$
    By the condition $(f_4)$, we get 
    $$\lim_{|(k,l)|\rightarrow\infty}I_u(k,l)=-\infty.$$
    Therefore, $I_u$ admits a global maximum at some $(k_0,l_0)\in [0,\infty)\cross[0,\infty)$. We first prove that $k_0, l_0 > 0$ by violating the following three cases.\\
    $(a)$ $k_0=l_0=0,$\\
    $(b)$ $k_0>0$, $l_0=0,$\\
    $(c)$ $k_0=0$, $l_0>0$.\\
    Let $l_0=0$, then $\Phi(k_0u^+)\geq \Phi(k u^++lu^-),~\forall\,k,l\geq 0$. 
    %In particular, $\Phi(k_0u^+)\geq \Phi(k u^+),~\forall\,k> 0$. So $k_0u^+$ is the maximum of $\Phi(ku^+)$. 
    We obtain, $\langle \Phi'(k_0u^+),k_0u^+\rangle =0$, that is
    \begin{align}\label{eq5.42}
        k_0^p\|u^+\|_{\mathbb{X}_0^{s,p}(\Omega)}=\int_\Omega f(x,k_0u^+) k_0u^+ \d x.
    \end{align}
    Since $u\in \mathcal{M}$, from Lemma \ref{lmn5.1}, we get $\langle \Phi'(u^+),u^+ \rangle<0$, implying that
    \begin{align}\label{eq5.43}
        \|u^+\|_{\mathbb{X}_0^{s,p}(\Omega)}<\int_\Omega f(x,u^+) u^+ \d x.
    \end{align}
    Now using the inequalities \eqref{eq5.42} and \eqref{eq5.43}, we obtain 
    \begin{align}\label{eq5.44}
        0< \int_\Omega \left[ \frac{f(x,u^+)}{|u^+|^{p-2}u^+}-\frac{f(x,k_0u^+)}{|k_0u^+|^{p-2}k_0u^+}\right]|u^+|^p \d x.
    \end{align}
    Moreover, using $(f_5)$ and \eqref{eq5.44}, we obtain $k_0\leq 1$. We know $\mathcal{H}(.,t)\geq 0$, $\mathcal{H}(.,t)$ is increasing in $(0,+\infty)$ and decreasing in $(-\infty, 0)$. Therefore, we have
    \begin{align}
        I_u(k_0,0)& =\Phi(k_0u^+) \nonumber\\
        &=\Phi(k_0u^+)-\frac{1}{p}\langle \Phi'(k_0u^+),k_0u^+\rangle \nonumber\\
        &=\frac{1}{p}\int_\Omega \mathcal{H}(x,k_0u^+)\d x=\frac{1}{p}\int_{\Omega^+}\mathcal{H}(x,k_0u^+)\d x\nonumber\\
        &\leq\frac{1}{p}\int_{\Omega^+}\mathcal{H}(x,u^+)\d x\nonumber\\
        &<\frac{1}{p} \bigg[ \int_{\Omega^+}\mathcal{H}(x,u^+)\d x+\int_{\Omega^-}\mathcal{H}(x,u^-)\d x\bigg]=\frac{1}{p}\int_{\Omega}\mathcal{H}(x,u)\d x\nonumber\\
        &=\Phi(u)-\frac{1}{p}\langle \Phi'(u),u\rangle =\Phi(u)=I_u(1,1), \nonumber
    \end{align}
    which gives a contradiction. Thus, $l_0>0$. By similar arguments, we conclude that $k_0>0$. Now following the arguments as in Lemma \ref{lmn5.4}, we get that $(1,1)$ is the unique critical point of $I_u$ in $(0,\infty)\cross(0,\infty)$. Hence, $I_u$ have global maximum at $(1,1)$, that is 
\begin{align}
    \Phi(u)=I(1,1)>I_u(k,l)=\Phi(k u^++lu^-),~\forall\,k,l>0~\text{such that}~(k,l)\neq(1,1).
\end{align}
 This completes the proof.
\end{proof}

\begin{lemma}\label{lmn5.7}
    If $\Phi(u^*)=m_s$ for some $u^*\in \mathcal{M}$, then $\Phi'(u^*)=0.$
\end{lemma}
\begin{proof}
    We prove this by the method of contradiction. Suppose $\Phi'(u^*)\neq0$. Then there exists $\rho_1,\mu_1>0$ such that
    $$\|\Phi'(u^*)\|_{\mathbb{X}_0^{s,p}(\Omega)}\geq\rho_1,~\forall ~u\in B_{3\mu_1}(u^*),$$
    where $B_{3\mu_1}(u^*)=\{u\in{\mathbb{X}_0^{s,p}(\Omega)}:\|u-u^*\|_{\mathbb{X}_0^{s,p}(\Omega)}\leq3\mu_1\}$. As $u^*\in \mathcal{M}$, we have $\langle \Phi'(u),u^\pm\rangle=0$ and $u^\pm\neq0$. Choose sufficiently small $\mu_1>0$ such that $u^\pm\neq0$ for all $u\in B_{3\mu_1}(u^*)$. Let $D=(1-\delta_1,1+\delta_1)\cross (1-\delta_1,1+\delta_1)$ with $\delta_1\in(0,\frac{1}{2})$ such that $k(u^*)^++l(u^*)^-\in B_{3\mu_1}(u^*)$, $\forall\,(k,l)\in \Bar{D}$. From Lemma \ref{lmn5.6}, we obtain 
    $$\bar{m}_s:=\max_{(k,l)\in \partial D}\Phi(k(u^*)^++l(u^*)^-)<m_s.$$
    Choose $\epsilon_1=\min\{\frac{m_s-\bar{m}_s}{2},\frac{\rho_1\mu_1}{8}\}$. Therefore, by similar arguments as in Lemma \ref{lmn4.9} (see also \cite[Lemma 2.3]{W1996}), it follows that there exists a continuous mapping $\eta:\mathbb{R}\cross {\mathbb{X}_0^{s,p}(\Omega)} \rightarrow {\mathbb{X}_0^{s,p}(\Omega)}$ such that \\
    $(a)$ $\eta(1,u)=u$ if $u\notin \Phi^{-1}[m_s-2\epsilon_1,m_s+2\epsilon_1]\cap B_{2\mu_1}(u^*)$,\\
    $(b)$ $\eta(1,\Phi^{m_s+\epsilon_1}\cap B_{\mu_1}(u^*))\subset \Phi^{m_s-\epsilon_1}$,\\
    $(c)$ $\Phi(\eta(1,u))\leq \Phi(u)$, $\forall\,u \in {\mathbb{X}_0^{s,p}(\Omega)}$.\\
    Define, $\sigma(k,l)=\eta(1,k(u^*)^++l(u^*)^-)$, $\forall\,(k,l)\in \bar{D}$. Thus, from Lemma \ref{lmn5.6} along with $(b)-(c)$ of the deformation lemma, we derive
    \begin{align}\label{eq5.46}
        \max_{(k,l)\in \bar{D}}\Phi(\sigma(k,l))=\max_{(k,l)\in \bar{D}}\Phi(\eta(k(u^*)^++l(u^*)^-))<m_s.
    \end{align}
    From \eqref{eq5.46}, we have $\{\sigma(k,l)\}_{(k,l)\in \bar{D}}\cap \mathcal{M}=\emptyset$. We will establish a contradiction by proving that $\{\sigma(k,l)\}_{(k,l)\in \bar{D}}\cap \mathcal{M}\neq\emptyset$. Now for any $(k,l)\in \bar{D}$, we define 
    \begin{align}\label{eq5.47}
        J_1(k,l)&=\left(\langle\Phi'(k(u^*)^++l(u^*)^-),(u^*)^+\rangle,\langle\Phi'(k(u^*)^++l(u^*)^-),(u^*)^-\rangle\right)~\text{and}\\ \label{eq5.48}
        J_2(k,l)&=\left(\frac{1}{k}\langle\Phi'(\sigma(k,l)),\sigma^+(k,l)\rangle, \frac{1}{l}\langle\Phi'(\sigma(k,l)),\sigma^-(k,l)\rangle\right).
    \end{align}
   Clearly, $J_1$ is $C^1$, since $f\in C^1(\bar{\Omega}\cross \mathbb{R},\mathbb{R})$. Therefore, $\langle \Phi'(u^*),(u^*)^\pm\rangle=0$, which implies that
    \begin{align}\label{eq5.49}
        &\int_\Omega |\nabla u^*|^{p-2} \nabla u^*\cdot\nabla (u^*)^+ \d x \nonumber\\
        &+ {\int_{\mathbb{R}^N} \int_{\mathbb{R}^N}} \frac{\splitfrac{|u^*(x)-u^*(y)|^{p-2}(u^*(x)-u^*(y))}{\times((u^*)^+(x)-(u^*)^+(y))}}{|x-y|^{N+sp}} \d x \d y = \int_\Omega f(x,(u^*)^+)(u^*)^+ \d x,\\ \label{eq5.50}
    & \int_\Omega |\nabla u^*|^{p-2} \nabla u^*\cdot\nabla (u^*)^- \d x \nonumber\\
    &+ {\int_{\mathbb{R}^N} \int_{\mathbb{R}^N}} \frac{\splitfrac{|u^*(x)-u^*(y)|^{p-2}(u^*(x)-u^*(y))}{\times((u^*)^-(x)-(u^*)^-(y))}}{|x-y|^{N+sp}} \d x \d y = \int_\Omega f(x,(u^*)^-)(u^*)^- \d x.
    \end{align}
   On using $(f_5)$, we have
    \begin{align}\label{eq5.51}
    \mathcal{H}'(x,s)s=f'(x,s)s^2-(p-1)f(x,s)s>0,~ \forall\,s\in \mathbb{R}\setminus \{0\}.
    \end{align}
    Let us denote,
    \begin{align}
        \alpha_1=&\int_\Omega |\nabla u^*|^{p-2} |\nabla (u^*)^+|^2 \d x \nonumber\\
        &+ {\int_{\mathbb{R}^N} \int_{\mathbb{R}^N}} \frac{|u^*(x)-u^*(y)|^{p-2}|((u^*)^+(x)-(u^*)^+(y))|^2}{|x-y|^{N+sp}} \d x \d y,\nonumber\\
        \alpha_2=& \int_\Omega f'_u(x,(u^*)^+)|(u^*)^+|^2 \d x,\nonumber\\
        \alpha_3=& \int_\Omega f(x,(u^*)^+)(u^*)^+ \d x,\nonumber\\
        \beta_1=&\int_\Omega |\nabla u^*|^{p-2} |\nabla (u^*)^-|^2 \d x \nonumber \\
        &+{\int_{\mathbb{R}^N} \int_{\mathbb{R}^N}} \frac{|u^*(x)-u^*(y)|^{p-2}|((u^*)^-(x)-(u^*)^-(y))|^2}{|x-y|^{N+sp}} \d x \d y,\nonumber\\
        \beta_2=&\int_\Omega f'_u(x,(u^*)^-)|(u^*)^-|^2 \d x,\nonumber\\
        \beta_3=&\int_\Omega f(x,(u^*)^-)(u^*)^- \d x,\nonumber\\
        \gamma_1=&\int_\Omega |\nabla u^*|^{p-2} \nabla (u^*)^-\cdot\nabla (u^*)^+ \d x \nonumber\\
        &+{\int_{\mathbb{R}^N} \int_{\mathbb{R}^N}} \frac{|u^*(x)-u^*(y)|^{p-2}((u^*)^-(x)-(u^*)^-(y))((u^*)^+(x)-(u^*)^+(y))}{|x-y|^{N+sp}} \d x \d y, \nonumber\\
        \gamma_2=&\int_\Omega |\nabla u^*|^{p-2}  \nabla (u^*)^+\cdot\nabla(u^*)^-\d x \nonumber\\
        &+{\int_{\mathbb{R}^N} \int_{\mathbb{R}^N}} \frac{|u^*(x)-u^*(y)|^{p-2}((u^*)^+(x)-(u^*)^+(y))((u^*)^-(x)-(u^*)^-(y))}{|x-y|^{N+sp}} \d x \d y. \nonumber
    \end{align}
    On using the inequalities \eqref{eq5.49}, \eqref{eq5.50} and \eqref{eq5.51}, we obtain
    $$\alpha_1>0,~\alpha_2>(p-1)\alpha_3>0,$$
     $$\beta_1>0,~\beta_2>(p-1)\beta_3>0,$$
     \begin{align}
         \gamma_1=&\int_\Omega |\nabla u^*|^{p-2} \nabla (u^*)^-\cdot\nabla (u^*)^+ \d x \nonumber\\
        &+{\int_{\mathbb{R}^N} \int_{\mathbb{R}^N}} \frac{|u^*(x)-u^*(y)|^{p-2}(-(u^*)^-(x)(u^*)^+(y)-(u^*)^-(y)(u^*)^+(x))}{|x-y|^{N+sp}} \d x \d y \nonumber \\
        =& {\int_{\mathbb{R}^N} \int_{\mathbb{R}^N}} \frac{|u^*(x)-u^*(y)|^{p-2}(-(u^*)^-(x)(u^*)^+(y)-(u^*)^-(y)(u^*)^+(x))}{|x-y|^{N+sp}} \d x \d y=\gamma_2>0,\nonumber
     \end{align}
     $$\alpha_1+\gamma_1=\alpha_3,~\beta_1+\gamma_2=\beta_3.$$
     Thus, we get
    { \begin{align}
          \det(J'_1(1,1))=&\langle\Phi''(u^*)(u^*)^+,(u^*)^+\rangle.\langle\Phi''((u^*)(u^*)^-,(u^*)^-\rangle \nonumber\\
         &- \langle\Phi''(u^*)(u^*)^+,(u^*)^-\rangle.\langle\Phi''((u^*)(u^*)^-,(u^*)^+\rangle \nonumber\\
         =&[(p-1)\alpha_1-\alpha_2)].[(p-1)\beta_1-\beta_2)]-(p-1)^2\gamma_1.\gamma_2 \nonumber\\
         >&(p-1)^2\gamma_1.\gamma_2-(p-1)^2\gamma_1.\gamma_2=0. \nonumber
     \end{align} }
     Hence, by the Brouwer degree theory, we obtain $\deg(J_1,D,0)=1$. In addition, from \eqref{eq5.46}, we have $\sigma(k,l)=k(u^*)^++l(u^*)^-),~\forall\,(k,l)\in \partial D$.\\
     Therefore, $$\deg(J_2,D,0)=\deg(J_1,D,0)=1.$$
     Thus, there exists $(k_0,l_0)\in D$ such that $J_2(k_0,l_0)=0$. By using the conditions $(a)$ and $(b)$ of $\eta$, we derive that
     $$u_0=\sigma(k_0,l_0)=\eta(1,k_0(u^*)^++s_0(u^*)^-)\in B_{3\mu_1}(u^*).$$
     Therefore, we get $\langle \Phi'(u_0),u_0^+\rangle=0=\langle \Phi'(u_0),u_0^-\rangle$ with $u^\pm\neq0$, i.e $u_0\in \{\sigma(k,l)\}_{(k,l)\in \bar{D}}\cap \mathcal{M}$. Thus, we arrive at a contradiction. Hence $u^*$ is a critical point of $\Phi$ and is a least energy sign-changing solution to the problem \eqref{MP}. This completes the proof.
\end{proof}
\begin{lemma}\label{lmn5.8}
    For any $u\in \mathcal{M}$, there exist $\bar{k}_u$, $\bar{l}_u$, $\in (0,1]$ such that $\bar{k}_u u^+$ and $\bar{l}_uu^-\in\mathcal{N}$.
\end{lemma}
\begin{proof}
    We only show that there exists $\bar{k}_u\in (0,1]$ such that $\bar{k}_u u^+\in\mathcal{N}$. The proof of $\bar{l}_u u^-\in\mathcal{N}$ follows from analogous arguments. Since $u\in \mathcal{M}$, we have $\langle \Phi'(u),u^{+} \rangle=0$, that is 
    \begin{align}\label{eq5.52}
        \|u^+\|_{\mathbb{X}_0^{s,p}(\Omega)}^p<\int_\Omega f(x,u^+)u^+ \d x=A^+(u).
    \end{align}
    Again, from Lemma \ref{lmn5.2}, there exists $\bar{k}_u>0$ such that $\bar{k}_uu^+\in \mathcal{N}$, and
    \begin{align}\nonumber
      \langle \Phi'(\bar{k}_uu^+),\bar{k}_uu^+  \rangle=0,
    \end{align}
    which implies
    \begin{align}\label{eq5.53}
        \bar{k}_u^p\|u^+\|_{\mathbb{X}_0^{s,p}(\Omega)}^p=\int_\Omega f(x,\bar{k}_uu^+)\bar{k}_uu^+ \d x.
    \end{align}
    Thus, using \eqref{eq5.52} and \eqref{eq5.53}, we obtain
    \begin{align}\label{eq5.54}
         0< \int_\Omega \left[ \frac{f(x,u^+)}{|u^+|^{p-2}u^+}-\frac{f(x,\bar{k}_uu^+)}{|\bar{k}_uu^+|^{p-2}\bar{k}_uu^+}\right]|u^+|^p \d x.
    \end{align}
    In addition, using $(f_5)$ and \eqref{eq5.54}, we derive $\bar{k}_u\leq 1$. This completes the proof. 
\end{proof}
We will now provide the second important theorem.

\section*{{\bf\emph{Proof of the Theorem \ref{T2.4}}}}
\begin{proof}
The existence of sign-changing solution is an immediate consequence of Lemma \ref{lmn5.3} and Lemma \ref{lmn5.5}. Again, Lemma \ref{lmn5.5} and Lemma \ref{lmn5.7} confirm that the functional $\Phi$ admits a critical point $u^*\in \mathcal{M}$ and it is a least energy sign-changing solution of \eqref{MP}. We know that $\mathcal{H}(.,t)$ increases in $(0,+\infty)$ and decreases in $(-\infty,0)$. Therefore, from Lemma \ref{lmn5.8}, we deduce 
\begin{align}
    m_s=&\Phi(u^*)=\Phi(u^*)-\frac{1}{p}\langle\Phi'(u^*),u^*\rangle\nonumber\\
    =&\frac{1}{p}\int_\Omega \mathcal{H}(x,u^*) \d x \nonumber\\
    =&\frac{1}{p} \bigg[ \int_{\Omega^+} \mathcal{H}(x,(u^*)^+) \d x+\int_{\Omega^-} \mathcal{H}(x,(u^*)^-) \d x\bigg]\nonumber\\
    >&\frac{1}{p} \bigg[ \int_{\Omega^+} \mathcal{H}(x,\bar{k}_{u^*}(u^*)^+) \d x+\int_{\Omega^-} \mathcal{H}(x,\bar{l}_{u^*}(u^*)^-) \d x\bigg]\nonumber\\
    =&\bigg[\Phi(\bar{k}_{u^*}(u^*)^+)-\frac{1}{p}\langle\Phi'(\bar{k}_{u^*}(u^*)^+),\bar{k}_{u^*}(u^*)^+\rangle\bigg]\nonumber \\
    &+\bigg[\Phi(\bar{l}_{u^*}(u^*)^-)-\frac{1}{p}\langle\Phi'(\bar{l}_{u^*}(u^*)^-),\bar{l}_{u^*}(u^*)^-\rangle\bigg]\nonumber\\
    =& \Phi(\bar{k}_{u^*}(u^*)^+)+\Phi(\bar{l}_{u^*}(u^*)^-)\geq 2c_s, \nonumber
\end{align}
that is, the energy level for the least energy of sign-changing solutions is strictly greater than twice that of the ground-state energy. This completes the proof.
\end{proof}

\begin{remark} We mention here that the results of Theorem \ref{T2.3} and Theorem \ref{T2.4} remain valid even if we consider a nonlocal operator with a generalized kernel, given by
$$\mathcal{L}_Ku=C({N,s,p})\text{P.V}\int_{\mathbb R^N} (|u(x)-u(y)|^{p-2})(u(x)-u(y)) K(x,y) \d y,$$
     where the above integral is defined as the principal value, $C({N,s,p})$ is a normalizing constant and $K(x,y)$ is a symmetric kernel such that
     \begin{equation}\label{gen k}
         \frac{\lambda^{-1}}{|x-y|^{N+ps}}\leq K(x,y) \leq \frac{\lambda}{|x-y|^{N+ps}},
     \end{equation}
    for some constant $\lambda\geq1.$ Note that, using the estimate \eqref{gen k}, one can obtain the lemmas above to prove Theorem \ref{T2.3} and Theorem \ref{T2.4}.
\end{remark}
\iffalse
{It is important to observe that with a generalized kernel, we could similarly show Lemma \ref{lmn4.1} to Lemma \ref{lmn4.7}, as the majority of these proofs involve establishing inequalities. Furthermore, we can show the $(PS)_c$ condition for a generalized kernel. Similarly, we can show Lemma \ref{lmn4.9} and Lemma \ref{lmn4.10} using the same explanation for a generalized kernel. Consequently, this concludes Theorem \ref{T2.3}. \\
Subsequently, the inequalities established in Lemma \ref{lmn5.1} are applicable to the generalized kernel. By applying the same argumentation from Lemma \ref{lmn5.2} to Lemma \ref{lmn5.8}, we can establish the proof for the generalized kernel. Thus, this ends Theorem \ref{T2.4}.} {check}
\fi

\section*{Conflict of interest statement}
\noindent On behalf of the authors, the corresponding author states that there is no conflict of interest.
 \section*{Data availability statement}
\noindent Data sharing does not apply to this article as no dataset were generated or analysed during the current study.

\section*{Acknowledgement}
\noindent Souvik Bhowmick would like to thank the Council of Scientific and Industrial Research (CSIR), Govt. of India for the financial assistance to carry out this research work [grant no. 09/0874(17164)/2023-EMR-I]. SG acknowledges the research facilities available at the Department of Mathematics, NIT Calicut under the DST-FIST support, Govt. of India [project no. SR/FST/MS-1/2019/40 Dated. 07.01.2020].

%\bibliographystyle{plain}
%\bibliography{svik}

\end{document}